\documentclass[journal]{IEEEtran}
\usepackage{algorithmic}
\usepackage{algorithm}
\usepackage{epsfig,amsfonts,color}
\usepackage{amssymb,url}
\usepackage{amsmath,amsthm}
\usepackage{mathtools}
\usepackage{tikz}
\usepackage{cite}

\usetikzlibrary{arrows.meta}

\usepackage[colorlinks=true,linkcolor=blue,citecolor=blue,urlcolor=blue]{hyperref}
\usepackage{url}
\newtheorem{definition}{Definition}

\newtheorem{theorem}{Theorem}
\newtheorem{lemma}{Lemma}

\newtheorem{proposition}{Proposition}

\newcommand{\bx}{\boldsymbol{x}}
\newcommand{\bu}{\boldsymbol{u}}

\newcommand{\bz}{\boldsymbol{z}}
\newcommand{\bp}{\boldsymbol{p}}

\newcommand{\blambda}{\boldsymbol{\lambda}}

\newcommand{\bd}{\boldsymbol{d}}
\newcommand{\bdelta}{\boldsymbol{\delta}}

\newcommand{\bG}{\boldsymbol{G}}

\newcommand{\bH}{\boldsymbol{H}}

\newcommand{\st}{\mathop{\text{\normalfont s.t.}}}
\newcommand{\diag}{\mathop{\text{\normalfont diag}}}
\begin{document}

\title{Diffusing-Horizon Model Predictive Control}

\author{Sungho Shin and Victor M. Zavala
\thanks{Sungho Shin and Victor M. Zavala are with Department of Chemical and Biological Engineering, University of Wisconsin-Madison, Madison, WI, 53706 USA e-mail: sungho.shin@wisc.edu, victor.zavala@wisc.edu.}}

\maketitle

\begin{abstract}
  We analyze a time-coarsening strategy for model predictive control (MPC) that we call diffusing-horizon MPC. This strategy seeks to overcome the computational challenges associated with optimal control problems that span multiple timescales. The coarsening approach uses a time discretization grid that becomes exponentially more sparse as one moves forward in time. This design is motivated by a recently established property of optimal control problems that is known as exponential decay of sensitivity. This property states that the impact of a parametric perturbation at a future time decays exponentially as one moves backward in time. We establish conditions under which this property holds for a constrained MPC formulation with linear dynamics and costs. Moreover, we show that the proposed coarsening scheme can be cast as a parametric perturbation of the MPC problem and thus the exponential decay condition holds. We use a heating, ventilation, and air conditioning plant case study with real data to demonstrate the proposed approach. Specifically, we show that computational times can be reduced by two orders of magnitude while increasing the closed-loop cost by only 3\%.
\end{abstract}


%
\IEEEpeerreviewmaketitle

\section{Introduction}

There is a growing need for model predictive control (MPC) formulations that can span multiple timescales\cite{baldea2007control,barrows2014time,falcone2007predictive,jackson2003temporal,dowling2017multi}. For instance, in energy system applications, one often needs to make planing decisions over months while capturing high-frequency (second-to-second) disturbances. Tractability issues arise in multiscale applications because of the need to use fine time discretization grids and/or long horizons. Strategies to deal with such issues include hierarchical decomposition, time-domain decomposition, sensitivity-based approximations, and time coarsening. 

Hierarchical decomposition deals with tractability issues by creating a family of controllers of different complexity. This approach typically uses low-complexity formulations that are obtained by reducing the model or by making assumptions such as periodicity. As a result, most of the existing hierarchical schemes do not have optimality guarantees (do not solve the original problem) \cite{scattolini2009architectures,picasso2010mpc,kumar2018hierarchical}. In time-domain decomposition, one partitions the horizon, and a coordination scheme is used to enforce convergence to the optimal solution. This approach is scalable and provides optimality guarantees but slow convergence rates are often observed  \cite{boyd2011distributed,geoffrion1972generalized,shin2019parallel,kozma2015benchmarking,giselsson2013accelerated}.  Sensitivity-based approximation strategies deal with tractability issues by computing a single Newton-type step at every sampling time. These approaches provide optimality guarantees but their scalability is limited by linear algebra (which is in turn limited by the time resolution and horizon used in formulating the MPC problem) \cite{diehl2002real,ohtsuka2004continuation,zavala2009advanced,diehl2009efficient,zavala2010real}.  

Time coarsening constructs an approximate representation of the problem by using a grid of sparse resolution (see Fig. \ref{fig:strategies}). In \cite{zavala2016new}, an architecture that combines hierarchical and time decomposition is proposed.  Here, hierarchical controllers of different complexity are obtained by using different time discretization resolutions. This approach is inspired by multi-grid schemes, which have been used for the solution of partial differential equations \cite{borzi2009multigrid}. In the context of multigrid, it has been observed that grid coarsening can be seen as a variable aggregation strategy and this observation has been used to derive more general {\it algebraic coarsening} schemes \cite{brandt2000general,notay2010aggregation}. In these schemes, one can aggregate variables and constraints in a flexible manner by exploiting the underlying algebraic structure (e.g., networks) \cite{shin2018multi,shin2019hierarchical}. The move-blocking MPC strategy reported in \cite{cagienard2007move,gondhalekar2010least} can be regarded as a special type of algebraic coarsening (this aggregates controls but not states), although the authors did not explicitly mention such a connection. 

In this work, we seek to mitigate tractability issues of MPC by using a time-coarsening strategy that we call {\em diffusing-horizon MPC}. Here, the original MPC problem is {\it coarsened} by applying a variable and constraint aggregation scheme over a coarse grid. The key design feature of the coarsening strategy is that the grid becomes exponentially more sparse as one moves forward in time (we call this exponential coarsening). This design is justified by a recently established property of optimal control problems which is known as {\it exponential decay of sensitivity} (EDS) \cite{xu2018exponentially,na2020exponential}. In technical terms, this property states that the effect of {\color{blue} parametric (data)} perturbations at a future time decays exponentially as one moves backward in time. In practical terms, EDS indicates that information in the far future has a negligible impact on the current control action. This intuitive property has been observed empirically in many applications but a precise theoretical characterization of this phenomenon has only been reported recently. Existing work has established that EDS holds for linear quadratic problems under a strong convexity assumption \cite{xu2018exponentially}. EDS has also been shown to hold for nonlinear problems in the vicinity of a solution that satisfies {\color{blue} the strong second-order conditions} (which can be induced by using convex stage costs) \cite{na2020exponential}. Intuitively, EDS holds under such settings because convex stage costs tend to naturally damp the effect of perturbations. {\color{blue} In this work, we establish conditions under which EDS holds for MPC formulations with linear dynamics and linear costs (as those encountered in economic MPC). The challenge here is that the problem is a linear program (LP)  and thus there is no concept of second order conditions (there is no concept of curvature); as such, a different type of analysis is needed. Notably, our theoretical results indicate that EDS can be established by characterizing the constraint set.} Moreover, we show that time-coarsening strategies can be cast as parametric perturbations. This allows us to use EDS to characterize the effect of coarsening {\color{blue} and to justify the proposed exponentially sparsifying time grid.} {\color{blue} A exponential coarsening approach was recently proposed in the context of continuous-time linear-quadratic optimal control problems \cite{grune2020exponential,grune2020abstract}. Here, exponential decay of sensitivity is established and this was used to show decay of the impact of the discretization error. The analysis provided in that work does not handle inequality constraints and requires quadratic costs. We also highlight that exponential coarsening has been explored in the literature and is used in commercial MPC products (see \cite{tan2016model} and references therein). Our contribution is a theoretical justification of why exponential coarsening works.} We use a case study for a heating, ventilation, and air conditioning (HVAC) system to show that the diffusing-horizon MPC strategy can yield dramatic reductions in computational time with small sacrifices in optimality.  

The paper is organized as follows: In Section \ref{sec:settings}, we introduce the nomenclature and problem settings. In Section \ref{sec:eds}, we establish EDS for the problem of interest. In Section \ref{sec:crs}, we introduce coarsening methods, show how to cast these as parametric perturbations, and use this notion to analyze how the effect of coarsening propagates in time. In this section we also present several coarsening strategies, including the diffusing-horizon scheme. In Section \ref{sec:cstudy}, we demonstrate our developments using an HVAC central plant control problem.

\begin{figure*}[t]
  \centering
  \begin{tikzpicture}
    \foreach \x in {1,...,30}
    \node[circle,fill,scale=.4]  at ((\x/2-1/2,0) {};
    \node at (0,0) {$|$};
    \node at (14.5,0) {$|$};
    \draw (0,0)--(14.5,0);
    \foreach \x in {1,4,7,10,13,16,19,22,25,28,30}
    \node[circle,fill,scale=.4]  at ((\x/2-1/2,-1) {};
    \node at (0,-1) {$|$};
    \node at (14.5,-1) {$|$};
    \draw (0,-1)--(14.5,-1);
    \foreach \x in {1,...,10,30}
    \node[circle,fill,scale=.4]  at ((\x/2-1/2,-2) {};
    \node at (0,-2) {$|$};
    \node at (14.5,-2) {$|$};
    \draw (0,-2)--(14.5,-2);
    \foreach \x in {1,2,3,4,5,6,7,11,15,21,30}
    \node[circle,fill,scale=.4]  at ((\x/2-1/2,-3) {};
    \node at (0,-3) {$|$};
    \node at (14.5,-3) {$|$};
    \draw (0,-3)--(14.5,-3);
    \node at (0,-3.5) {$1$};
    \node at (14.5,-3.5) {$N$};
  \end{tikzpicture}
  \caption{Schematic representation of different coarsening strategies. From top to bottom: full resolution (no coarsening), equal-spacing scheme, full-then-sparse scheme, diffusing-horizon (exponential coarsening) scheme.}\label{fig:strategies}
\end{figure*}

\section{Problem Setting}\label{sec:settings}

\subsection{Notation}
The set of real numbers and the set of integers are denoted by $\mathbb{R}$ and $\mathbb{I}$, respectively, and we define $\mathbb{I}_{M:N}:=\mathbb{I}\cap[M,N]$, $\mathbb{I}_{>0}:=\mathbb{I}\cap(0,\infty)$, $\mathbb{R}_{>0}:=(0,\infty)$, $\mathbb{I}_{\geq 0}:=\mathbb{I}\cap[0,\infty)$, $\mathbb{R}_{\geq 0}:=[0,\infty)$, and $\overline{\mathbb{R}}:=\mathbb{R}\cup\{-\infty,\infty\}$.
Absolute values of real numbers and cardinality of sets are denoted by $|\cdot|$. Euclidean norms of vectors and induced-Euclidean norms of matrices are denoted by$\Vert\cdot\Vert$. By default, we treat vectors as column vectors and we use the syntax:
\begin{align*}
  \{v_i\}_{i=m}^n=\begin{bmatrix}v_{m}\\ \vdots\\ v_{{n}}\end{bmatrix},\{a_{i,j}\}_{i,j=m}^n=\begin{bmatrix}a_{m,m}&\cdots&a_{{m},{n}}\\ \vdots&\ddots&\vdots\\ a_{{n},{m}}&\cdots&a_{{n},{n}}\end{bmatrix},
\end{align*}
and, more generally,
\begin{align*}
  \{v_i\}_{i\in \mathcal{I}}=\begin{bmatrix}v_{i_1}\\ \vdots\\ v_{i_{m}}\end{bmatrix},\{a_{i,j}\}_{i\in \mathcal{I},j\in \mathcal{J}}=\begin{bmatrix}a_{i_1,j_1}&\cdots&a_{i_{1},j_{n}}\\ \vdots&\ddots&\vdots\\ a_{i_{m},j_{1}}&\cdots&a_{i_{m},j_{n}}\end{bmatrix},
\end{align*}
where $\mathcal{I}=\{i_1<\cdots<i_{m}\}$ and $\mathcal{J}=\{j_1<\cdots<j_{n}\}$;
we use $v[i]$ to denote the $i$th component of $v$; $M[i,j]$ to denote the $[i,j]$th component of $M$; $v[\mathcal{I}]:=\{v[i]\}_{i\in \mathcal{I}}$; $M[\mathcal{I},\mathcal{J}]:=\{M[i,j]\}_{i\in \mathcal{I},j\in \mathcal{J}}$; $(v_1,v_2,\cdots,v_n):=[v_1^\top\;v_2^\top\;\cdots\;v_n^\top]^\top$; and 
\begin{align*}
  \diag(M_1,\cdots,M_n):=\begin{bmatrix}
  M_1\\
  &\ddots\\
  &&M_n
  \end{bmatrix}.
\end{align*}
For matrices $A,B$, notation $A\succeq B$ indicates that $A-B$ is positive semi-definite while $A\geq B$ represents a component-wise inequality. The smallest and largest singular value of matrix $M$ are denoted by $\sigma_{\min}(M)$ and $\sigma_{\max}(M)$, respectively. The line segment between two vectors $d,d'$ are denoted by $[d,d']:=\{(1-s)d+sd':s\in[0,1]\}$. We call a collection $\{X_1,\cdots,X_n\}$ of subsets of set $X$ a partition of $X$ if the following conditions are satisfied: (i) $\bigcup_{i=1}^n X_i = X$, (ii) $X_i\neq \emptyset$ for any $i\in\mathbb{I}_{1:n}$, and (iii) $X_i\cap X_j=\emptyset$, for any $i,j\in\mathbb{I}_{1:n}$ and $i\neq j$. {\color{blue} We denote $(\cdot)_+:=\max(0,\cdot)$.}

\subsection{MPC Problem}
We consider an MPC problem of the form:
\begin{subequations}\label{prob:mpc}
  \begin{align}
    \min_{\{x_i,u_i\}_{i=1}^N}\;&\sum_{i=1}^N q_i^\top x_i + r_i^\top u_i\\
    \st\;    &x_1 = v_1\\
    &x_{i}= {A}_{i-1}x_{i-1}+B_{i-1}u_{i-1}+ v_{i},\;i\in\mathbb{I}_{2:N}\label{prob:mpc-con-1}\\
    &E_ix_i+F_iu_i=w_i,\;i\in\mathbb{I}_{1:N}\label{prob:mpc-con-2}\\
    &\underline{x}_i\leq x_i\leq \overline{x}_i,\; i\in\mathbb{I}_{1:N}\label{prob:mpc-con-3}\\
    &\underline{u}_i\leq u_i\leq \overline{u}_i,\; i\in\mathbb{I}_{1:N}\label{prob:mpc-con-4}.
  \end{align}
\end{subequations}
Here, $x_i\in\mathbb{R}^{n_x}$ are the state variables; $u_i\in\mathbb{R}^{n_u}$ are the control variables; the {\em parameters} $q_i\in\mathbb{R}^{n_x}$, $r_i\in\mathbb{R}^{n_u}$, $v_i\in\mathbb{R}^{n_x}$, and $w_i\in\mathbb{R}^{n_w}$ are {\color{blue} the problem {\em data} (e.g., disturbance forecasts and costs)}. Variable $v_1\in\mathbb{R}^{n_x}$ can be regarded as the initial condition. Symbols $A_i\in\mathbb{R}^{n_x\times n_x}$ and $B_i\in\mathbb{R}^{n_x\times n_u}$ denote the time-varying system matrices; $E_i\in\mathbb{R}^{n_w\times n_x}$ and $F_i\in\mathbb{R}^{n_w\times n_u}$ are constraint matrices; $\underline{x}_i,\overline{x}_i\in\mathbb{R}^{n_x}$ are lower and upper bounds of state variables; $\underline{u}_i,\overline{u}_i\in\mathbb{R}^{n_u}$ are lower and upper bounds for control variables, and $N\in\mathbb{I}_{>0}$ is the horizon length.  Constraints \eqref{prob:mpc-con-1}-\eqref{prob:mpc-con-2} allow us to accommodate models with dynamic and algebraic constraints and initial and terminal conditions. Problem \eqref{prob:mpc} is a linear program and has been widely studied in the context of economic MPC  \cite{dowling2017multi,kumar2018hierarchical,kumar2019stochastic,risbeck2019economic}. {\color{blue} In this work we focus on the computational tractability of deterministic MPC formulations. Issues associated with uncertainties (e.g., forecast errors) are assumed to be dealt with via feedback (as in a traditional MPC approach).}

Problem \eqref{prob:mpc} can be written in the following compact form:
\begin{subequations}\label{prob:lp-detail}
  \begin{align}
    \min_{\{z_i\}_{i=1}^N}\;&\sum_{i=1}^Np_i^\top z_i\\
    \st\;&G_{i,i-1}z_{i-1}+G_{i,i} z_{i} \geq d_i,\;i\in\mathbb{I}_{1:N}\label{prob:lp-detail-dyn}
  \end{align}
\end{subequations}
Here, $z_i\in\mathbb{R}^{{n}}$ are the primal variables; $p_i\in\mathbb{R}^{n}$ are the cost vectors;  $d_i\in\mathbb{R}^{{m}}$ are the data vectors; and $G_{i,j}\in\mathbb{R}^{{m}\times {n}}$ are the constraint mapping matrices. {\color{blue} We denote the dual variables for \eqref{prob:lp-detail-dyn} as $\lambda_i\in\mathbb{R}^{m}$.} Here, we let $G_{1,0}:=0$ and $z_0=0$ (for convenience). One can derive \eqref{prob:lp-detail} from \eqref{prob:mpc} by defining, for $i\in\mathbb{I}_{1:N}$, $z_i:=(x_i,u_i)$, $p_i:=(q_i,r_i)$, $d_i:=(v_{i},-v_{i},w_{i},-w_{i},\underline{x}_i,-\overline{x}_i,\underline{u}_i,-\overline{u}_i)$,
\begin{align*}
  G_{i,i-1}&:=
  \begin{bmatrix}
    -A_{i-1}&-B_{i-1}\\
    A_{i-1}&B_{i-1}\\
    &\\
    &\\
    &\\
    &\\
    &\\
    &
  \end{bmatrix},
  G_{i,i}:=\begin{bmatrix}
  {I}_{n_x}&\\
  -{I}_{n_x}&\\
  E_i & F_i \\
  -E_i & -F_i \\
  {I}_{n_x}&\\
  -{I}_{n_x}&\\
  &{I}_{n_u}\\
  &- {I}_{n_u}
  \end{bmatrix},
\end{align*}
$m:=4n_x+2n_u+2n_w$, and $n:=n_x+n_u$. In our analysis we will also use the compact form: 
\begin{subequations}\label{prob:lp}
  \begin{align}
    \min_{z_{}}\;& \bp^\top \bz_{}\\
    \st\;
    &\bG\bz_{} \geq \bd,\label{prob:lp-pon}
  \end{align}
\end{subequations}
where $\bz\in\mathbb{R}^{nN}$ is the decision variable vector, $\bp\in\mathbb{R}^{nN}$ is the cost vector, $\bG\in\mathbb{R}^{mN\times nN}$ is the constraint mapping matrix, $\bd\in\mathbb{R}^{mN}$ is the data vector, $\blambda\in\mathbb{R}^{mN}$ is the dual variable vector.
One can derive \eqref{prob:lp} from \eqref{prob:lp-detail} by defining $\bz:=\{z_i\}_{i=1}^N$, $\blambda:=\{\lambda_i\}_{i=1}^N$, $\bd:=\{d_i\}_{i=1}^N$, $\bp:=\{p_i\}_{i=1}^N$, and $\bG:=\{G_{i,j}\}_{i,j=1}^N$ (for convenience we let $G_{i,j}=0$ unless $0\leq i-j\leq 1$). We denote Problem \eqref{prob:lp} (equivalently, \eqref{prob:mpc} or \eqref{prob:lp-detail}) as a parametric optimization problem $P(\bd; \bG,\bp)$ or simply $P(\bd)$. We observe that the feasible set of $P(\bd;\bG,\bp)$ is compact and that $\bG[\mathbb{I}_{m(i-1)+1:mi},\mathbb{I}_{n(j-1)+1:nj}] \neq 0$ only if $0\leq i-j \leq 1$. Vectors and matrices that are associated with more than one time index are denoted by boldface letters (e.g., $\bz$, $\bG$, $\bp$, and $\bd$).

\section{Exponential Decay of Sensitivity}\label{sec:eds}
In this section we study the sensitivity of solutions $\bz^*(\bd)=(z^*_1(\bd),\cdots,z^*_N(\bd))$ of $P(\bd)$ with respect to data perturbations in $\bd=(d_1,\cdots,d_N)$. In other words, we aim to address the question {\it how does the solution $\bz^*(\bd)$ changes as the parameter (data) $\bd$ changes?} Studies on sensitivity of LPs go back to the pioneering work of Robinson \cite{robinson1973bounds}, where the author estimated a solution sensitivity bound against perturbations on cost vectors, constraint coefficient matrices, and constraint right-hand side vectors. Robinson used Hoffman's theorem as the main technical tool to establish these results \cite{hoffman2003approximate}. Similar results have also been established in \cite{schrijver1998theory}, where an upper bound of the Lipschitz constant of the solution mapping is estimated. Here we aim to obtain a similar result but we are interested in deriving {\it {\color{blue} stage-wise}} sensitivity bounds, under the context of the MPC problem \eqref{prob:mpc}. Specifically, in our main theorem, we establish how the solution at a certain time stage changes in the face of a perturbation at another time stage (upper bounds of {\color{blue} stage-wise} Lipschitz constants are also derived). This perturbation analysis setting is illustrated in Fig. \ref{fig:eds}. In our analysis, we use a fundamentally different approach than those reported in the literature; specifically, we establish sensitivity bounds for the {\it basic solutions} of the LP and then establish sensitivity bounds for the optimal solutions around such basic solutions. This approach allows us to handle non-unique solutions (the solution set is not a singleton), which is a key difference over existing sensitivity results for convex quadratic and nonlinear programs \cite{xu2018exponentially,na2020exponential}. 

\begin{figure*}[t]
  \centering
  \begin{tikzpicture}
    \foreach \x in {1,...,30}
    \node[circle,fill,scale=.4]  at ((\x/2-1/2,0) {};
    \node at (0,0) {$|$};
    \node at (14.5,0) {$|$};
    \node at (0,-.5) {$1$};
    \node at (14.5,-.5) {$N$};
    \node at (4.5,-.5) {$i$};
    \node at (9.5,-.5) {$j$};
    \node at (4.5,.5) {$z^*_i(\bd)$};
    \node at (9.5,.5) {$d_j$};
    \node at (7,-1.5) {$|i-j|$};
    \draw[<->] (4.5,-1)--(9.5,-1);
    \draw (0,0)--(14.5,0);
  \end{tikzpicture}
  \caption{Illustration of perturbation analysis setting.}\label{fig:eds}
\end{figure*}
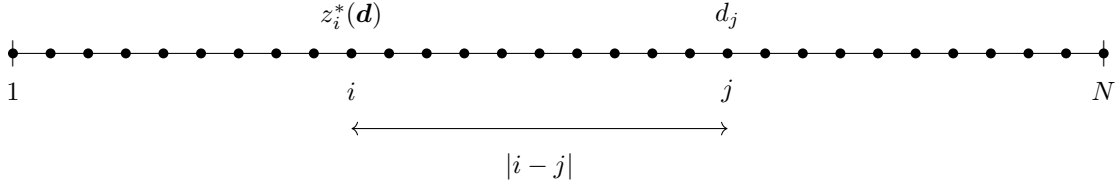

\subsection{Preliminaries}
In this section we revisit basic concepts of linear programming that will be used for our analysis. 
\begin{definition}[Set of admissible data]
  The set of admissible data $\mathbb{D}\subseteq \mathbb{R}^{mN}$ is defined as: 
  \begin{align*}
    \mathbb{D}:=\left\{\bd\in\mathbb{R}^{mN}: P(\bd) \text{ is feasible.} \right\}
  \end{align*}
\end{definition}
\begin{proposition}\label{prop:conv}
  The set $\mathbb{D}$ is convex.
\end{proposition}
\begin{proof}
  Let $\bd,\bd'\in\mathbb{D}$, and $\bz,\bz'\in\mathbb{R}^{{nN}}$ be feasible solutions of $P(\bd)$ and $P(\bd')$, respectively. From feasibility, we have that $\bG\bz\geq \bd$ and $\bG\bz'\geq \bd'$. By taking a linear combination, we have that $\bG((1-s)\bz+s\bz') \geq (1-s)\bd+s\bd'$ holds for any $s\in[0,1]$. Thus, $P((1-s)\bd+s\bd')$ is feasible for any $s\in[0,1]$ and this implies $(1-s)\bd+s\bd'\in \mathbb{D}$ for any $s\in[0,1]$. Therefore, the set $\mathbb{D}$ is convex.
\end{proof}
\begin{proposition}\label{prop:existence}
  A solution of $P(\bd)$ exists for any $\bd\in\mathbb{D}$.
\end{proposition}
\begin{proof}
  A solution exists because $P(\bd)$ has a continuous objective and a nonempty feasible set (since $\bd\in\mathbb{D}$) that is also compact (the boundedness comes from the inspection of \eqref{prob:mpc-con-3}-\eqref{prob:mpc-con-4}).
\end{proof}

\begin{definition}[Basis, basic solution, basic feasible solution, and basic optimal solution]
  Consider $P(\bd)$ with $\bd\in\mathbb{D}$. 
  \begin{itemize}
  \item[(a)] A set $B\subseteq\mathbb{I}_{1:mN}$ is called a basis of $P(\bd)$ if $\bG[B,:]$ is non-singular.
  \item[(b)] A vector $\bz^B(\bd)\in\mathbb{R}^{nN}$ is called a basic solution of $P(\bd)$ associated with basis $B$ if the following holds:
    \begin{align}\label{eqn:bsol}
      \bG[B,:]\bz^B(\bd)= \bd[B].
    \end{align}
  \item[(c)] A basis $B$ is called feasible for $P(\bd)$, and $\bz^B(\bd)$ is called a basic feasible solution if $\bz^B(\bd)$ is feasible to $P(\bd)$.
  \item[(d)] A basis $B$ is called optimal for $P(\bd)$, and $\bz^B(\bd)$ is called a basic optimal solution if $\bz^B(\bd)$ is optimal to $P(\bd)$.
  \end{itemize}
\end{definition}
The optimality of a basis is often defined by checking if the associated reduced cost is non-negative (see  \cite{bertsimas1997introduction}). A basic solution can have a negative component in the reduced cost even if it is optimal in the case where it is degenerate. Thus, the non-negativity of the reduced cost is a sufficient condition (but not necessary) for optimality. In this paper we call a basis optimal if the associated basic solution is optimal. We also note that, if a set $B\subseteq\mathbb{I}_{1:mN}$ is a basis for some $P(\bd)$ with $\bd\in\mathbb{D}$, then $B$ is a basis for $P(\bd')$ for any $\bd'\in\mathbb{D}$. Thus, we will write that $B$ is a basis for $P(\cdot)$. Here, we observe that the basic solution mapping $\bz^B:\mathbb{R}^{mN}\rightarrow \mathbb{R}^{nN}$ is linear with respect to data $\bd$.

\begin{definition}[Optimal bases]
  The collection $\mathbb{B}_{D}$ of optimal bases for $D\subseteq\mathbb{D}$ is defined as:
  \begin{align*}
    \mathbb{B}_{D}&:=\left\{B\subseteq\mathbb{I}_{1:mN}:\exists \bd\in D \text{ s.t. }B \text{ is optimal for }P(\bd)\right\}
  \end{align*}
  Furthermore, we let $\mathbb{B}:=\mathbb{B}_{\mathbb{D}}$.
\end{definition}

\subsection{Perturbation and Basis}
We aim to analyze the path of $\bz^*(\bd)$ (not necessarily unique) with respect to the perturbation on $\bd$. In particular, we consider data points $\bd,\bd'\in\mathbb{D}$ and study how the solution changes along the line segment $[\bd,\bd']$. We use the notation:
\begin{align*}
  \bd^s:=(1-s)\bd+s\bd',\; s\in[0,1].
\end{align*}
We use basic solutions in order to estimate the change of the optimal solutions. The following theorem provides a useful connection between basic and optimal solutions. {\color{blue} Specifically, the analysis of the optimal solution is difficult due to the absence of the explicit relationship between the solution and the data. However, the basic solution has an explicit relationship with the data (given by \eqref{eqn:bsol}); as such, by establishing the connection between the basic solution and optimal solution, we can characterize the effect of the data on the optimal solution.}

\begin{theorem}\label{thm:basic}
  Consider $\bd,\bd'\in\mathbb{D}$. There exists a finite sequence of real numbers $\{s_0=0<\cdots<s_{N_d}=1\}$ and a finite sequence of bases $\{B_1,\cdots,B_{N_d}\in\mathbb{B}_{[\bd,\bd']}\}$ such that the following holds.
  \begin{itemize}
  \item[(a)] For $k\in\mathbb{I}_{1:N_d}$, $B_k$ is optimal for $P(\bd^s)$ with any $s\in[s_{k-1},s_k]$.
  \item[(b)] For $k\in\mathbb{I}_{1:N_d-1}$, $\bz^{B_k}(\bd^{s_k})=\bz^{B_{k+1}}(\bd^{s_k})$.
  \end{itemize}
\end{theorem}
To prove Theorem \ref{thm:basic}, we first establish the following technical lemmas.
\begin{lemma}\label{lem:sub-1}
  There exists a finite sequence of real numbers $\{s_0=0<\cdots<s_{N_d}=1\}$ and a finite sequence of nonempty collections of bases $\{\mathbb{B}_1,\cdots,\mathbb{B}_{N_d}\subseteq\mathbb{B}_{[\bd,\bd']}\}$  such that the following holds for $k\in\mathbb{I}_{1:N_d}$.
  \begin{itemize}
  \item[(a)] If $B_k\in\mathbb{B}_k$, then $B_k$ is optimal for $P(\bd^s)$ with any $s\in[s_{k-1},s_k]$.
  \item[(b)] If $B_k$ is an optimal basis for $P(\bd^s)$ with some $s\in(s_{k-1},s_k)$, then $B_k\in\mathbb{B}_k$.
  \end{itemize}
\end{lemma}
\begin{proof}
  By the convexity of $\mathbb{D}$ (Proposition \ref{prop:conv}), we have that $[\bd,\bd']\subseteq \mathbb{D}$. Since the feasible set of $P(\bd)$ cannot contain a line (due to the compactness of the feasible set), $P(\bd)$ with $\bd\in\mathbb{D}$ has at least one basic feasible solution \cite[Theorem 2.6]{bertsimas1997introduction}. Furthermore, by the existence of a basic feasible solution and the existence of a solution (Proposition \ref{prop:existence}), there exists a basic optimal solution for any $\bd\in\mathbb{D}$ \cite[Theorem 2.7]{bertsimas1997introduction}. Let $\mathbb{B}_{[\bd,\bd']}=\{B_{(1)},\cdots,B_{(L)}\}$ (note that $\mathbb{B}_{[\bd,\bd']}$ is a finite set, since it is a subset of a power set of finite set). Consider the objective value mapping for the basic solutions $\pi_{(\ell)}:[0,1]\rightarrow\overline{\mathbb{R}}$ for $\ell\in\mathbb{I}_{1:L}$ defined by:
  \begin{align*}
    \pi_{(\ell)}(s) = 
    \begin{cases}
      \bp^\top \bz^{B_{(\ell)}}(\bd^s)& B_{(\ell)} \text{ is feasible to }P(\bd^s),\\
      +\infty&\text{otherwise}.
    \end{cases}    
  \end{align*}
  From the existence of basic optimal solution for $\bd\in\mathbb{D}$, the optimal value mapping of $P(\bd^s)$ with respect to $s$ can be represented as the mapping $\pi:[0,1]\rightarrow\overline{\mathbb{R}}$ defined by:
  \begin{align*}
    \pi(s) := \min_{\ell=1,\cdots,L}\pi_{(\ell)}(s).
  \end{align*}
  From the existence of a solution (Proposition \ref{prop:existence}), we have that $\pi(s)<\infty$ on $[0,1]$. Due to the piece-wise linear nature of $\pi(\cdot)$, we have that $\Pi_{(\ell)}:=\{s\in[0,1]:\pi_{(\ell)}(s)=\pi(s)\}$ for $\ell\in\mathbb{I}_{1:L}$ are obtained as {\it finite unions of closed intervals}. By collecting all the end points, we construct the set:
  \begin{align*}
    \Pi:=\bigcup_{\ell=1}^L \Pi_{(\ell)}\setminus \text{int}(\Pi_{(\ell)}),
  \end{align*}
  where $\text{int}(\cdot)$ denotes the interior set. Note that $\Pi$ is a finite set and that $0,1\in\Pi$. Now choose $\{s_0,\cdots,s_{N_d}\}:=\Pi$ and
  \begin{align*}
    \mathbb{B}_k:=&\left\{B_{(\ell)}\in\mathbb{B}_{[\bd,\bd']}:[s_{k-1},s_k]\subseteq\Pi_{(\ell)} \right\},\;k\in\mathbb{I}_{\color{blue} 1:N_d}.
  \end{align*}
  Since $\Pi$ is constructed from the endpoints of the collections of intervals whose union is equal to $[0,1]$, for each $[s_{k-1},s_k]$, there exists at least one interval, which is a subset of some $\Pi_{(\ell)}$, that contains $[s_{k-1},s_k]$. Thus, each $\mathbb{B}_k$ is nonempty.
  We can observe that for $k\in\mathbb{I}_{1:N_d}$,
  \begin{align*}
    \mathbb{B}_k\cap \{ B_{(\ell)}\in\mathbb{B}_{[\bd,\bd']}: (s_{k-1},s_k)\cap \Pi_{(\ell)} = \emptyset\}=\emptyset,
  \end{align*}
  since $[s_{k-1},s_k]$ is either a subset of some $\Pi_{(\ell)}$ or its interior has empty intersection with $\Pi_{(\ell)}$. We have thus constructed $\mathbb{B}_1,\cdots,\mathbb{B}_K$ in such a way that satisfy (a) and (b).
\end{proof}

\begin{lemma}\label{lem:sub-2}
The following holds for $k\in\mathbb{I}_{1:N_d-1}$; for any $B_k\in\mathbb{B}_k$, there exists $B_{k+1}\in\mathbb{B}_{k+1}$ such that:
  \begin{align*}
    \bz^{B_k}(\bd^{s_k})=\bz^{B_{k+1}}(\bd^{s_k}).
  \end{align*}
\end{lemma}
\begin{proof}
    We continue from Lemma \ref{lem:sub-1}; consider the set (not necessarily a basis) $\widehat{B}\subseteq \mathbb{I}_{1:mN}$ of constraint indices of $P(\bd^{s_k})$ that are active on $\bz^{B_k}(\bd^{s_k})$. We let $\widehat{B}^c:=\mathbb{I}_{1:mN}\setminus \widehat{B}$.
  
  We observe that, for any basis $B\in\mathbb{B}_{[\bd,\bd']}$ that satisfies $B\subseteq \widehat{B}$, $\bz^{B_k}(\bd^{s_k})=z^{B}(\bd^{s_k})$ holds due to the uniqueness of basic solutions (in the sense that there exists a unique basic solution {\color{blue} associated} with a specific basis). Accordingly,
  \begin{subequations}\label{eqn:conds}
    \begin{align}
      \bG[\widehat{B},:]\bz &= \bd[\widehat{B}]\\
      \bG[\widehat{B}^c,:]\bz &> \bd[\widehat{B}^c]\label{eqn:conds-ineq}
    \end{align}
  \end{subequations}
  holds for any $\bz=\bz^{B}(\bd^{s_k})$ with such $B$. Since the basic solution mapping $\bz^{B}(\bd^s)$ is affine with respect to $s$ and $B\in\mathbb{B}_{[\bd,\bd']}$ such that satisfies $B\subseteq\widehat{B}$ are finite, one can choose $s\in(s_k,s_{k+1})$ so that \eqref{eqn:conds-ineq} is satisfied for $\bz=\bz^B(\bd^s)$ with any basis $B\subseteq \widehat{B}$. We now define problem $P_{\widehat{B}}(\bd)$ as:
  \begin{subequations}
    \begin{align}
      \min_{\bz}\;& \bp^\top \bz\\
      \st\;&  \bG[\widehat{B},:]\bz \geq \bd[\widehat{B}].\label{eqn:constr-hat}
    \end{align}
  \end{subequations}
  From the KKT conditions of $P(\bd^{s_k})$, one can see that $\bz^{B_k}(\bd^{s_k})$ is a solution of $P_{\widehat{B}}(\bd^{s_k})$. From the dual feasibility of $P_{\widehat{B}}(\bd^s)$ (in the dual problem, only the cost vector changes with respect to the change in $\bd$) and weak duality, there exists a lower bound in the objective value of $P_{\widehat{B}}(\bd^s)$. Furthermore,  from primal feasibility (the feasible set of $P_{\widehat{B}}(\bd^s)$ is a superset of the feasible set of ${P}(\bd^s)$), there exists a solution of $P_{\widehat{B}}(\bd^s)$ \cite[Corollary 2.3]{bertsimas1997introduction}. Moreover, the existence of the basic feasible solution of ${P}(\bd^{s_k})$ and the nonemptyness of the feasible set of $P_{\widehat{B}}(\bd^s)$ imply the existence of the basic feasible solution of $P_{\widehat{B}}(\bd^s)$ \cite[Theorem 2.6]{bertsimas1997introduction}. By the existence of the optimal solution and the existence of the basic feasible solution, there exists a basic optimal solution of $P_{\widehat{B}}(\bd^s)$ \cite[Theorem 2.7]{bertsimas1997introduction}. We now choose the basis associated with such a basic optimal solution as $B_{k+1}$. Since it is a basis for $P_{\widehat{B}}(\cdot)$, we have that $B_{k+1}\subseteq \widehat{B}$. One can see that $\bz^{B_{k+1}}(\bd^s)$ is feasible to $P(\bd^s)$, since \eqref{eqn:constr-hat} is satisfied (by feasibility to $P_{\widehat{B}}(\bd^s)$) and \eqref{eqn:conds-ineq} is satisfied (recall that we chose $s$ such that \eqref{eqn:conds-ineq} is satisfied). Since $\bz^{B_{k+1}}(\bd^s)$ is an optimal solution of $P_{\widehat{B}}(\bd^s)$, which is a relaxed problem of ${P}(\bd^s)$, and it is feasible to ${P}(\bd^s)$, it is an optimal solution of ${P}(\bd^s)$. Therefore, $B_{k+1}$ is an optimal basis of $P(\bd^s)$. This implies that $B_{k+1}\in\mathbb{B}_{k+1}$ (Lemma \ref{lem:sub-1}, condition (b)). Moreover, recall that $B_{k+1}\subseteq \widehat{B}$. Accordingly, $\bz^{B_{k+1}}(\bd^{s_k})=\bz^{B_{k}}(\bd^{s_k})$ holds. 
\end{proof}

\begin{proof}[Proof of Theorem \ref{thm:basic}]
  We choose $\{0=s_0<\cdots<s_{N_d}=1\}$ to be the one that appears in Lemma \ref{lem:sub-1}. We choose some $B_1\in\mathbb{B}_1$ (recall that $\mathbb{B}_1$ is nonempty). Then we repeat the following for $k=2,\cdots,N_d-1$: for given $B_k$, choose $B_{k+1}\in\mathbb{B}_{k+1}$ such that $\bz^{B_k}(\bd^{s_k})=\bz^{B_{k+1}}(\bd^{s_k})$ holds (such a basis exists in $\mathbb{B}_{k+1}$ by Lemma \ref{lem:sub-2}). With this process, one can construct $\{B_1,\cdots,B_{N_d}\}$ in such a way that (a) and (b) hold.
\end{proof}

We draw the following conclusion from Theorem \ref{thm:basic}: for the perturbation path $[\bd,\bd']$, one can divide the path into a finite union of line segments $\{[\bd^{s_{k-1}},\bd^{s_k}]\subseteq[\bd,\bd']\}_{k=1}^K$ in which the optimal basis does not change and the associated basic solutions do not {\it jump} when the basis changes. Thus, the overall solution path is a union of paths where the individual paths follow that of basic solutions. Thus, in order to characterize solution sensitivity, it suffices to characterize the paths of the basic solutions. 

\subsection{Basic Solution Sensitivity}
We now study the sensitivity of the basic solution mappings; in particular, we aim to obtain {\color{blue} stage-wise} bounds of Lipschitz constants for such mappings. 
We define the following for basis $B$ of $P(\cdot)$, matrix $M\in\mathbb{R}^{nN\times nN}$, and vector $v\in\mathbb{R}^{nN}$:
  \begin{align*}
    M_{[i][j]}&:=M[\mathcal{I}_i,\mathcal{I}_j],&& M_{[i]\{j\}}:=M[\mathcal{I}_i,\mathcal{J}_j]\\
    M_{\{i\}[j]}&:=M[\mathcal{J}_i,\mathcal{I}_j],&& M_{\{i\}\{j\}}:=M[\mathcal{J}_i,\mathcal{J}_j]\\
    v_{[i]}&:=v[\mathcal{I}_i],&& v_{\{i\}}:=v[\mathcal{J}_i],
  \end{align*}
  {\color{blue} where $\mathcal{I}_i:=\mathbb{I}_{n(i-1)+1:ni}$, $\mathcal{J}_i:=\{\ell\in\mathbb{I}_{1:nN}:b_\ell\in\mathbb{I}_{m(i-1)+1:mi}\}$, and $B=\{b_1<\cdots<b_{nN}\}$. Note that $\{\mathcal{I}_i\}_{i=1}^N$ and $\{\mathcal{J}_i\}_{i=1}^N$ respectively partition $\mathbb{I}_{1:nN}$.}
  
  \vspace{0.1in}
  The following is the main result of this subsection.
\begin{theorem}\label{thm:key}
  Consider the data vectors $\bd=(d_1,\cdots,d_N)$ and $\bd'=(d'_1,\cdots,d'_N)\in\mathbb{D}$. The following holds for any basis $B\in\mathbb{B}$, the associated basic solution mapping $\bz^B(\cdot)=(z^B_1(\cdot),\cdots,z^B_N(\cdot))$, and $i\in\mathbb{I}_{1:N}$:
    \begin{align*}
    \Vert z^B_i(\bd)-z^B_i(\bd')\Vert\leq\sum_{j=1}^N  \Gamma_B\rho_B^{(|i-j|-1)_+}\Vert  d_j- d'_j\Vert.
  \end{align*}
    where $\Gamma_B:=\displaystyle\frac{\overline{\sigma}_B}{\underline{\sigma}_B^2}$, $\displaystyle\rho_B := \frac{\overline{\sigma}_B^2-\underline{\sigma}_B^2}{\overline{\sigma}_B^2+\underline{\sigma}_B^2}$, $\underline{\sigma}_B:=\sigma_{\min}(\bG[B,:])$, and $\overline{\sigma}_B:=\sigma_{\max}(\bG[B,:])$.
\end{theorem}
To prove Theorem \ref{thm:key}, we first prove the following.
\begin{lemma}\label{lem:claim}
  Let $B$ be a basis of $P(\cdot)$ and $\bH_B:=\bG[B,:]\bG[B,:]^\top$. If $|i-j|>k\in\mathbb{I}_{>0}$, the following holds:
  \begin{align}\label{eqn:claim}
    (\bH_B^k)_{\{i\}\{j\}}=0.
  \end{align}
\end{lemma}
\begin{proof}
  We proceed by induction: first we note that $\bG[B,:]_{\{i\}[j]}$ is a submatrix of $G_{i,j}$. Thus, $\bG[B,:]_{\{i\}[j]}\neq 0$ only if $0\leq i-j\leq 1$. From this observation, we have that if $|i-j|>1$, $(\bH_B)_{\{i\}\{j\}}=\sum_{\ell=1}^N \bG[B,:]_{\{i\}[\ell]} (\bG[B,:]_{\{j\}[\ell]})^\top = 0$, since either $0\leq i-\ell \leq 1$ or $0\leq j-\ell\leq 1$ are violated. Thus, \eqref{eqn:claim} holds for $k=1$. Suppose that \eqref{eqn:claim} holds for $k=k'$. We have that if $|i-j|>k'+1$, for any $\ell\in\mathbb{I}_{1:N}$, either $|i-\ell|>k'$ or $|j-\ell|>1$ holds. Accordingly, if $|i-j|>k'+1$, $(\bH_B^{k'+1})_{\{i\}\{j\}} = \sum_{\ell=1}^N (\bH_B^{k'})_{\{i\}\{\ell\}}(\bH_B)_{\{\ell\}\{j\}} = 0$.
\end{proof}
Lemma \ref{lem:claim} states that the block-banded structure of $\bH_B$ is preserved if we take powers of $\bH_B$. We are now ready to prove Theorem \ref{thm:key}.

\begin{proof}[Proof of Theorem \ref{thm:key}]
  From $\underline{\sigma}_B^2 I \preceq \bH_B\preceq  \overline{\sigma}_B^2I$, we have:
 \begin{align}\label{eqn:factor}
   \frac{\underline{\sigma}_B^2-\overline{\sigma}_B^2}{\underline{\sigma}_B^2+\overline{\sigma}_B^2} I\preceq I-\frac{2}{\underline{\sigma}_B^2+\overline{\sigma}_B^2}\bH_B
   \preceq \frac{-\underline{\sigma}_B^2+\overline{\sigma}_B^2}{\underline{\sigma}_B^2+\overline{\sigma}_B^2} .
  \end{align}
 Moreover, we have that:
 \begin{align*}
   \bH_B^{-1} &= \frac{2}{\underline{\sigma}_B^2+\overline{\sigma}_B^2} \left(\frac{2}{\underline{\sigma}_B^2+\overline{\sigma}_B^2}\bH_B \right)^{-1}\\
   &= \frac{2}{\underline{\sigma}_B^2+\overline{\sigma}_B^2} \left(I-(I-\frac{2}{\underline{\sigma}_B^2+\overline{\sigma}_B^2}\bH_B )\right)^{-1}\\
   &= \frac{2}{\underline{\sigma}_B^2+\overline{\sigma}_B^2} \sum_{k=0}^\infty \left(I-\frac{2}{\underline{\sigma}_B^2+\overline{\sigma}_B^2}\bH_B \right)^k,
 \end{align*}
 where the last equality follows from \cite[Theorem 5.6.9 and Corollay 5.6.16]{horn2012matrix} and \eqref{eqn:factor}. From $\bG[B,:]^{-1}=\bG[B,:]^\top \bH_B^{-1}$ we have:
  \begin{align}\label{eqn:key-1}
    \bG[B,:]^{-1}
    &= \frac{2}{\underline{\sigma}_B^2+\overline{\sigma}_B^2} \sum_{k=0}^\infty \bG[B,:]^\top \left(I-\frac{2}{\underline{\sigma}_B^2+\overline{\sigma}_B^2}\bH_B \right)^k.
  \end{align}
  We have that, if $|i-j|>k+1$, then:
  \begin{align}\label{eqn:sparsity}
    &\left(\bG[B,:]^\top\left(I-\frac{2}{\underline{\sigma}_B^2+\overline{\sigma}_B^2}\bH_B \right)^k\right)_{[i]\{j\}}\\\nonumber
    &=\sum_{\ell=1}^N \left(\left(\bG[B,:]\right)_{\{\ell\}[i]}\right)^\top \left(I-\frac{2}{\underline{\sigma}_B^2+\overline{\sigma}_B^2}\bH_B \right)^k_{\{\ell\}\{j\}}=0,
  \end{align}
  since either $\bG[B,:]_{\{\ell\}[i]}=0$ or $(I-\frac{2}{\underline{\sigma}_B^2+\overline{\sigma}_B^2}\bH_B )^k_{\{\ell\}\{j\}}=0$ by sparsity of $\bG[B,:]$ and Lemma \ref{lem:claim}.
  By extracting submatrices from \eqref{eqn:key-1}, one establishes the following.
  \begin{align*}
    &(\bG[B,:]^{-1})_{[i]\{j\}}\nonumber \\
    &= \frac{2}{\underline{\sigma}_B^2+\overline{\sigma}_B^2} \sum_{k=k_0}^\infty \left(\bG[B,:]^\top\left(I-\frac{2\bH_B}{\underline{\sigma}_B^2+\overline{\sigma}_B^2} \right)^k\right)_{[i]\{j\}}
  \end{align*}
  where $k_0:=\max(|i-j|-1,0)$ and the summation over $k=0,\cdots,k_0-1$ is neglected by \eqref{eqn:sparsity}. Using the triangle inequality and the fact that the matrix norm of a submatrix is always less than that of the original matrix, we have
  \begin{align}\label{eqn:inv}
    &\Vert (\bG[B,:]^{-1})_{[i]\{j\}} \Vert  \\
    &\leq \frac{2}{\underline{\sigma}_B^2+\overline{\sigma}_B^2} \sum_{k=k_0}^\infty \left\Vert \bG[B,:]^\top\left(I-\frac{2}{\underline{\sigma}_B^2+\overline{\sigma}_B^2}\bH_B\right)^k\right\Vert\nonumber\\
    &\leq \frac{2}{\underline{\sigma}_B^2+\overline{\sigma}_B^2} \sum_{k=k_0}^\infty \left\Vert \bG[B,:]\right\Vert\left\Vert I-\frac{2}{\underline{\sigma}_B^2+\overline{\sigma}_B^2} \bH_B\right\Vert^k\nonumber\\
    &\leq \frac{2}{\underline{\sigma}_B^2+\overline{\sigma}_B^2} \sum_{k=k_0}^\infty\overline{\sigma}_B\left(\frac{\overline{\sigma}_B^2-\underline{\sigma}_B^2}{\overline{\sigma}_B^2+\underline{\sigma}_B^2}\right)^k\nonumber\\
    &\leq \underbrace{\frac{\overline{\sigma}_B}{\underline{\sigma}_B^2}}_{\Gamma_B}\Bigg(\underbrace{\frac{\overline{\sigma}_B^2-\underline{\sigma}_B^2}{\overline{\sigma}_B^2+\underline{\sigma}_B^2}}_{\rho_B}\Bigg)^{(|i-j|-1)_+} \nonumber\\
    &\leq \Gamma_B \rho_B^{(|i-j|-1)_+}.\nonumber
  \end{align}
  where the second inequality follows from the submultiplicativity and transpose invariance of the induced 2-norm, the third inequality follows from \eqref{eqn:factor} and the fact that the induced 2-norm of symmetric matrix is equal to the largest magnitude eigenvalue, and the fourth inequality follows from the summation of geometric series, and the last inequality is obtained by simplification. 

  From \eqref{eqn:bsol} and the block multiplication formula, we have:
  \begin{align*}
    z^B_{i}( \bd)-z^B_{i}( \bd') &= \sum_{j=1}^N\left(\bG[B,:]^{-1}\right)_{[i]\{j\}}(\bd_{\{j\}}-\bd'_{\{j\}}).
  \end{align*}
  From the triangle inequality and the submultiplicativity of matrix norms, we have that:
  \begin{align*}
    \left\Vert z^B_{i}( \bd)-z^B_{i}( \bd')\right\Vert &\leq \sum_{j=1}^N\left\Vert\left(\bG[B,:]^{-1}\right)_{[i]\{j\}}\right\Vert\left\Vert \bd_{\{j\}}-\bd'_{\{j\}}\right\Vert\nonumber\\
    &\leq \Gamma_B \rho_B^{(|i-j|-1)_+}\left\Vert d_j-d'_j\right\Vert,
  \end{align*}
  where the second inequality follows from \eqref{eqn:inv} and the observation that $\bd_{\{j\}}$ is a subvector of $d_j$.
\end{proof}

Theorem \ref{thm:key} establishes that the sensitivity of the basic solution at a given time location against a data perturbation at another location decays exponentially with respect to the distance between such locations. Note that the coefficient $\Gamma_B$ and $\rho_B$ are obtained as functions of $\underline{\sigma}_B$ and $\overline{\sigma}_B$, which are the quantities associated with matrix $\bG[B,:]$.

\subsection{Exponential Decay of Sensitivity}
We are now ready to state our main result of this section, which we call exponential decay of sensitivity. Here, we connect Theorems \ref{thm:basic} and \ref{thm:key} to obtain a sensitivity bound for the optimal solutions.

\begin{theorem}[Exponential Decay of Sensitivity, EDS]\label{thm:eds}
  Consider $\bd=(d_1,\cdots,d_N),\bd'=(d'_1,\cdots,d'_N)\in \mathbb{D}$. There exists a solution $\bz^*(\bd)=(z^*_1(\bd),\cdots,z^*_N(\bd))$ of $P(\bd)$ and a solution $\bz^*(\bd')=(z^*_1(\bd'),\cdots,z^*_N(\bd'))$ of $P(\bd')$ such that the following holds for $i\in\mathbb{I}_{1:N}$:
  \begin{align}\label{eqn:eds}
    \Vert z^*_{i}( \bd)-z^*_{i}( \bd')\Vert\leq \sum_{j=1}^N \Gamma\rho^{(|i-j|-1)_+}\Vert d_j- d'_j\Vert,
  \end{align}
  where $\Gamma:=\displaystyle\frac{\overline{\sigma}}{\underline{\sigma}^2} $, $\displaystyle\rho := \frac{\overline{\sigma}^2-\underline{\sigma}^2}{\overline{\sigma}^2+\underline{\sigma}^2}$, $\displaystyle\underline{\sigma}:=\min_{B\in\mathbb{B}}\underline{\sigma}_B$, and $\displaystyle\overline{\sigma}:=\max_{B\in\mathbb{B}}\overline{\sigma}_B$.
\end{theorem}
\begin{proof}
  {\color{blue} Under Theorem \ref{thm:basic}, there exists a solution $\bz^*(\bd)$ of $P(\bd)$, a solution $\bz^*(\bd')$ of $P(\bd')$, a sequence $\{s_0=0<\cdots<s_{N_d}=1\}$, and a sequence of bases $\{B_1,\cdots,B_{N_d}\in\mathbb{B}_{[\bd,\bd']}\}$ such that}:
  \begin{align}\label{eqn:eds-1}
    z^*_i(\bd)- z^*_i(\bd') = \sum_{k=1}^{N_d} z^{B_k}_i(\bd^{s_{k-1}})-z^{B_k}_i(\bd^{s_k}).
  \end{align}
  By applying the triangle inequality we have that
  \begin{align*}
    &\Vert z^*_i(\bd)- z^*_i(\bd')\Vert 
    \leq \sum_{k=1}^{N_d} \Vert z^{B_k}_i(\bd^{s_{k-1}})-z^{B_k}_i(\bd^{s_k})\Vert\\\nonumber
    &\quad\leq \sum_{k=1}^{N_d} \sum_{j=1}^N (s_k-s_{k-1})\Gamma_{B_k}\rho_{B_k}^{(|i-j|-1)_+}\Vert d_j-d'_j\Vert,
  \end{align*}
  where the second inequality follows from Theorem \ref{thm:key}. Here, observe that $\Gamma_{B_k}\leq \Gamma$ and $\rho_{B_k}^{(|i-j|-1)_+} \leq \rho^{(|i-j|-1)_+}$ hold. In consequence,
  \begin{align*}
    &\Vert z^*_i(\bd)- z^*_i(\bd')\Vert \\
    &\quad\leq \sum_{j=1}^N \Gamma\rho^{(|i-j|-1)_+} \Vert d_j-d'_j\Vert\sum_{k=1}^{N_d}(s_k-s_{k-1})\\
    &\quad\leq \sum_{j=1}^N \Gamma\rho^{(|i-j|-1)_+} \Vert d_j-d_j'\Vert.\qedhere
  \end{align*}
\end{proof}
The coefficients $\Gamma \rho^{(|i-j|-1)_+}$ are {\color{blue} stage-wise} Lipschitz constants that represent the sensitivity of the solution $z_i^*(\bd)$ (at stage $i$) with respect to the perturbation on $d_j$ (at stage $j$). Since $\rho\in(0,1)$, we see that the solution sensitivity {\it decays} exponentially as the distance $|i-j|$ between time stages $i$ and $j$ increases.
\\

  The sensitivity decays at a substantial rate if $\rho$ is sufficiently small. In this context, we observe that $\rho$ can be made sufficiently small if there exist a sufficiently large lower bound in $\sigma_{\min}(\bG[B,:])$ and a sufficiently small upper bound in $\sigma_{\max}(\bG[B,:])$. Unfortunately, in many applications, the condition number of $\bG[B,:]$ tends to increase as the prediction horizon $N$ grows. Thus, in such a case, the decay rate $\rho$ deteriorates as the size of the problem grows. In order for the results in Theorem \ref{thm:eds} to be useful, there should exist uniform lower and upper bounds (which do not depend on the prediction horizon length $N$) in the singular values of $\bG[B,:]$. The remaining question is when do such uniform bounds exist. Existing works for linear quadratic settings suggest that controllability-related assumption is necessary to guarantee the existence of such uniform bounds \cite{xu2018exponentially,na2020exponential}. We leave this issue as a topic of future work.

\section{Algebraic Coarsening}\label{sec:crs}

In this section, we use the exponential decay of sensitivity to derive an algebraic coarsening technique for Problem \eqref{prob:lp-detail}.  Specifically, we are interested in deriving a coarsening strategy that is capable of reducing the dimension of the problem without sacrificing much solution quality. We call such a reduced-dimension problem as the {\it coarse problem} and its solution as the {\it coarse solution}. A key observation that arises in receding horizon control is that one does not need to maintain the overall quality of the solution, but one only needs to maintain the quality of the solution at the first time step (the control action implemented in the actual system). Thus, we design a coarsening strategy in a way that prioritizes the quality of $z^*_1$ over that of the remaining trajectory $\{z^*_i\}_{i=2}^{N}$. This is done by creating a grid that becomes exponentially more sparse as one marches forward in time.  In order to conduct this analysis, we will prove that a coarsening scheme can be cast as a parametric perturbation.

Move-blocking is a conventional method for dimensionality reduction of MPC problems \cite{cagienard2007move,gondhalekar2010least}. In this strategy, one {\it blocks} control variables (i.e., one forces a block of control variables to move together) to reduce the number of degrees of freedom. Here blocking can be regarded as a special case of a coarsening strategy that only aggregates control variables over time subdomains. While the reduction of degrees of freedom is highly desired in certain settings (e.g., in explicit MPC \cite{borrelli2017predictive}), one often sees better scalability by reducing the total number of variables and constraints. In the proposed approach, we {\it aggregate} the variables and constraints (equivalently, aggregate primal and dual variables) to reduce their numbers. This approach is known as algebraic coarsening; algebraic coarsening strategies were introduced to optimization problems in \cite{shin2018multi}, where the authors used the strategy to design a hierarchical architecture to solve large optimization problems. Here we present a generalized version of the algebraic coarsening scheme that incorporates prior information of the primal-dual solution. A schematic illustration of the algebraic coarsening approach is shown in Fig. \ref{fig:coarsening}.

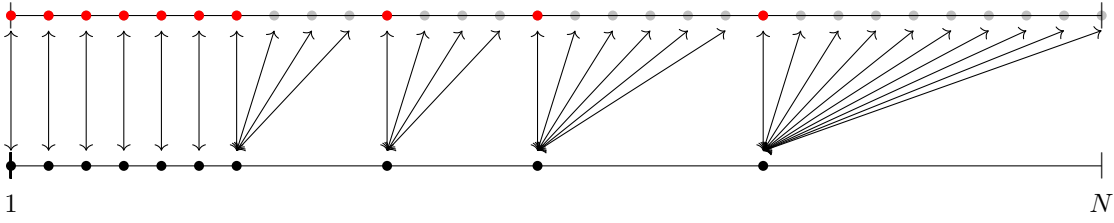
\begin{figure*}[t]
  \centering
  \begin{tikzpicture}
    \foreach \x in {1,...,30}
    \node[circle,fill,lightgray,scale=.4]  at ((\x/2-1/2,0) {};
    \node at (0,0) {$|$};
    \node at (14.5,0) {$|$};
    \draw (0,0)--(14.5,0);
    \foreach \x in {1,2,3,4,5,6,7,11,15,21}
    \node[circle,fill,red,scale=.4]  at ((\x/2-1/2,0) {};

    \foreach \x in {1,...,30}
    \node at (0,-2) {$|$};
    \node at (14.5,-2) {$|$};
    \draw (0,-2)--(14.5,-2);
    \foreach \x in {1,2,3,4,5,6,7,11,15,21}
    \node[circle,fill,scale=.4]  at ((\x/2-1/2,-2) {};

    \foreach \x in {1,2,3,4,5,6,7,11,15,21}
    \draw[<->] (\x/2-1/2,-.2)--(\x/2-1/2,-1.8);
    \foreach \x in {8,12,16,22}
    \draw[<->] (\x/2-1/2,-.2)--(\x/2-2/2,-1.8);
    \foreach \x in {9,13,17,23}
    \draw[<->] (\x/2-1/2,-.2)--(\x/2-3/2,-1.8);
    \foreach \x in {10,14,18,24}
    \draw[<->] (\x/2-1/2,-.2)--(\x/2-4/2,-1.8);
    \foreach \x in {19,25}
    \draw[<->] (\x/2-1/2,-.2)--(\x/2-5/2,-1.8);    
    \foreach \x in {20,26}
    \draw[<->] (\x/2-1/2,-.2)--(\x/2-6/2,-1.8);
    \draw[<->] (27/2-1/2,-.2)--(27/2-7/2,-1.8);
    \draw[<->] (28/2-1/2,-.2)--(28/2-8/2,-1.8);
    \draw[<->] (29/2-1/2,-.2)--(29/2-9/2,-1.8);
    \draw[<->] (30/2-1/2,-.2)--(30/2-10/2,-1.8); 
    \node at (0,-2.5) {$1$};
    \node at (14.5,-2.5) {$N$};
  \end{tikzpicture}
  \caption{Illustration of algebraic coarsening strategy (diffusing-horizon).}\label{fig:coarsening}
\end{figure*}

\subsection{Coarse Problem Formulation}

Consider a {\it coarse grid} $\{M_1=1<\cdots<M_K\}\subseteq \mathbb{I}_{1:N}$. The coarse grid uniquely defines a partition $\{\mathbb{I}_{M_k:N_k}\}_{k=1}^K$ of the time grid $\mathbb{I}_{1:N}$, where $N_k=M_{k+1}-1$ (we let $M_{K+1}:=N+1$ for convenience). We define $L_k:=|\mathbb{I}_{M_k:N_k}|$. Using this coarse grid, we define the {\it coarsening operators}.
\begin{definition}[Coarsening operators]\label{def:crs}
  The operators $T\in\mathbb{R}^{nN\times nK}$ and $U\in\mathbb{R}^{mN\times mK}$
  associated with the coarse grid $\{M_k\}_{k=1}^K$ are defined as
  \begin{align*}
    T:=\diag(T_1,T_2,\cdots,T_K),\; U:=\diag(U_1,U_2,\cdots,U_K)
  \end{align*}
  where
  \begin{align*}
    T_k := L_k^{-1/2}\begin{bmatrix}
      {I}_{n}\\\vdots\\{I}_{n}
    \end{bmatrix}_{nL_k \times n},\;
    U_k := L_k^{-1/2}\begin{bmatrix}
      {I}_{m}\\\vdots\\{I}_{m}
    \end{bmatrix}_{mL_k \times m}.
  \end{align*}
\end{definition}
Note that the coarsening operators can be fully determined from the choice of coarse grid $\{M_k\}_{k=1}^K$. Thus, the design of coarsening strategy reduces to {\it designing a partitioning scheme} for the time grid $\mathbb{I}_{1:N}$. We note that the operators $T$, $U$, $T_k$, and $U_k$ are orthogonal and that  $T,U,T_k,U_k\geq 0$.

Now we discuss the coarse problem formulation; we start by rewriting $P(\bd;\bG,\bp)$ as:
\begin{subequations}\label{prob:lp-pre-coarse}
\begin{align}
  &\min_{\{\bz_k\}_{k=1}^K}\;\sum_{k=1}^N \bp_k^\top \bz_k\\
  \quad&\st\;\bG_{k,k-1}\bz_{k-1}+\bG_{k,k} \bz_{k} \geq \bd_k,\;k\in\mathbb{I}_{1:K}\label{prob:lp-pre-coarse-con}
\end{align}
\end{subequations}
where we use the syntax:
\begin{align*}
  \boldsymbol{M}_{k,\ell} &:= \{M_{i,j}\}_{i\in \mathbb{I}_{M_k:N_k},j\in \mathbb{I}_{M_\ell:N_\ell}},\;
  \boldsymbol{v}_k &:= \{v_i\}_{i\in \mathbb{I}_{M_k:N_k}},
\end{align*}
and assume $\bG_{1,0}=0$ and $\bz_0=0$ for convenience. {\color{blue} We denote dual variables for \eqref{prob:lp-pre-coarse-con} by $\blambda_k$.} Here, note that $\bz_k\neq z_k$, $\bG_{k,k}\neq G_{k,k}$, and so on. Now we will apply coarsening to \eqref{prob:lp-pre-coarse} (or equivalently, $P(\bd;\bG,\bp)$) to formulate the coarse problem. The coarse problem with grid $\{M_k\}_{k=1}^K$ is formulated as:
\begin{subequations}\label{prob:lp-coarse-1}
  \begin{align}
    \min_{\{\widetilde{z}_k\}_{k=1}^K}\;&\sum_{k=1}^K\widetilde{p}_k^\top \widetilde{z}_k\\
    \st\;& \widetilde{G}_{k,k-1} \widetilde{z}_{k-1}+\widetilde{G}_{k,k} \widetilde{z}_k \geq \widetilde{d}_k, \label{eqn:lp-coarse-1-con}
  \end{align}
\end{subequations}
where:
\begin{align*}
  \widetilde{p}_k&=T_k^\top \left(\boldsymbol{p}_k - \boldsymbol{G}^\top_{k,k} \boldsymbol{\lambda}^o_k-\boldsymbol{G}^\top_{k+1,k} \boldsymbol{\lambda}^o_{k+1}\right)\\
  \widetilde{d}_k&=U^\top_k \left(\boldsymbol{d}_k-\boldsymbol{G}_{k,k-1}\boldsymbol{z}^o_{k-1}-\boldsymbol{G}_{k,k}\boldsymbol{z}^o_{k}\right);
\end{align*}
$\widetilde{G}_{k,k-1}=U_k^\top \boldsymbol{G}_{k,k-1} T_{k-1}$; $\widetilde{G}_{k,k}=U_k^\top \boldsymbol{G}_{k,k} T_{k}$; we let $\widetilde{G}_{0,1}=0$, $\widetilde{z}_{0}=0$, $\bz^o_0=0$, $\blambda^o_{K+1}=0$, $\bG_{1,0}=0$, and $\bG_{K+1,K}=0$ for convenience. {\color{blue} We denote dual variables for \eqref{eqn:lp-coarse-1-con} by $\blambda_k$.} Observe that the overall structure of problem \eqref{prob:lp-pre-coarse} is preserved in \eqref{prob:lp-coarse-1} but the corresponding variables do not have the same dimensionality ($\widetilde{z}_k\in\mathbb{R}^n\leftrightarrow \bz_k\in\mathbb{R}^{nL_k}$ and $\widetilde{\lambda}_k\in\mathbb{R}^m\leftrightarrow \blambda_k\in\mathbb{R}^{mL_k}$)

Coarse problem \eqref{prob:lp-coarse-1} can be equivalently written as:
\begin{subequations}\label{prob:lp-coarse-0}
  \begin{align}
    \min_{\widetilde{\bz}}\;&\widetilde{\bp}^\top \widetilde{\bz}\\
    \st\;& \widetilde{\bG} \widetilde{\bz} \geq \widetilde{\bd} \quad(\widetilde{\blambda}),
  \end{align}
\end{subequations}
where $\widetilde{\bp} = \{\widetilde{p}_k\}_{k=1}^K$, $\widetilde{\bd} = \{\widetilde{d}_k\}_{k=1}^K$, $\widetilde{\bG} = \{\widetilde{G}_{k,\ell}\}_{k,\ell=1}^K$ (we assume $\widetilde{G}_{k,\ell}=0$ unless $0\leq k-\ell\leq 1$). We observe that the following holds:
$\widetilde{\bp}= T^\top (\bp - \bG^\top \lambda^o)$, $\widetilde{\bd}= U^\top (\bd- \bG\bz^o )$, and $\widetilde{\bG}= U^\top \bG T$.
We represent \eqref{prob:lp-coarse-0} (or equivalently \eqref{prob:lp-coarse-1}) as $P(\widetilde{\bd};\widetilde{\bG},\widetilde{\bp})$. The definition of $\widetilde{\bp},\widetilde{\bG},\widetilde{\bd}$ depends on the choice of the coarse grid $\{M_k\}_{k=1}^K$ and the choice of the {\it primal-dual prior guess} $(\bz^o,\blambda^o)$. The primal-dual solutions of the coarse problem are denoted by $(\widetilde{\bz}^*,\widetilde{\blambda}^*)$ and are called the {\it coarse solutions}. The solution of $P(\widetilde{\bd},\widetilde{\bG},\widetilde{\bp})$ can be projected to the full space to obtain the solution on the full domain. Such a projection can be performed as $\bz' = \bz^o + T \widetilde{\bz}^*$ and  $\blambda' = \blambda^o + U \widetilde{\blambda}^*$. We call $(\bz',\blambda')$ the {\it projected coarse solution}. By inspecting the block diagonal structure of coarsening operators, one can see that each variable in $P(\widetilde{\bd},\widetilde{\bG},\widetilde{\bp})$ has upper and lower bounds. We can thus see that the feasible set of $P(\widetilde{\bd},\widetilde{\bG},\widetilde{\bp})$ is compact.

An intuitive justification for the formulation of \eqref{prob:lp-coarse-0} is as follows. First, we have aggregated the variables by replacing $\bz_k$ {\color{blue} by $\bz^o_k+T_k\widetilde{z}_k$}. By inspecting the structure of $T_k$, one can see that this manipulation reduces the number of variables. Furthermore, we take linear combinations of the constraints by left multiplying $U^\top_k$ on the constraint functions associated with $k$. Along with this, the $(\blambda^o_k)^\top \left[\bG_{k,k-1} T_{k-1}\widetilde{\bz}_{k-1}+ \bG_{k,k} T_k\widetilde{\bz}_k \right]$ terms are subtracted from the objective function to reflect the prior guess of the dual variables. By inspecting the structure of $U_k$, one can see that this manipulation reduces the number of constraints, and effectively aggregates the dual variables. Further justification for the formulation of \eqref{prob:lp-pre-coarse} will be given later after introducing additional nomenclature.

A formal justification for the formulation of \eqref{prob:lp-coarse-0} arises from its consistency with full problem \eqref{prob:lp}; here, we show that if one has a perfect prior guess of the primal-dual solution (i.e., $\bz^o=\bz^*$ and $\blambda^o=\blambda^*$), the projected coarse solution is equal to the solution of the original problem (i.e., $\bz^o=\bz'$ and $\blambda^o=\blambda'$). We formally state this result as the following proposition

\begin{proposition}\label{prop:cons}
  Consider $P(\bd;\bG,\bp)$ and $P(\widetilde{\bd};\widetilde{\bG},\widetilde{\bp})$ and use $(\bz^o,\blambda^o)$ to denote the primal-dual solution of $P(\bd;\bG,\bp)$. We have that $(\widetilde{\bz},\widetilde{\blambda})=0$ is a primal-dual solution of $P(\widetilde{\bd};\widetilde{\bG},\widetilde{\bp})$.
\end{proposition}
\begin{proof}
  From the KKT conditions of $P(\bd;\bG,\bp)$ at $(\bz^o,\blambda^o)$, we have:
  \begin{align*}
    &\bp - \bG^\top \blambda^o= 0,
    &&\bG\bz^o - \bd \geq0\\
    &\blambda^o \geq 0,
    &&\diag(\blambda^o) (\bG{\color{blue} \bz^o}-\bd) = 0.
  \end{align*}
  By left multiplying using $T^\top$ and $U^\top$ we have that
  \begin{align*}
    &T^\top \bp - (\bG T)^\top \blambda^o  = 0,
    &&  U^\top \bG\bz^o - U^\top \bd \geq 0
  \end{align*}
  hold (recall that $U\geq 0$). The KKT conditions of $P(\widetilde{\bd},\widetilde{\bG},\widetilde{\bp})$ state that:
  \begin{align}\label{eqn:kkt-2}
    \begin{aligned}
      T^\top \bp - (\bG T)^\top \blambda^o -\widetilde{\bG}^\top \widetilde{\blambda}&= 0\\
      \widetilde{\bG}\widetilde{\bz} + U^\top \bG\bz^o- U^\top \bd &\geq 0\\
      \widetilde{\blambda}\geq 0,\quad
      \diag(\widetilde{\blambda})(\bG\widetilde{\bz}-\widetilde{\bd})&=0.
    \end{aligned}
  \end{align}
  From \eqref{eqn:kkt-2}, one can show that if $(\widetilde{\bz},\widetilde{\blambda})=0$, the KKT conditions \eqref{eqn:kkt-2} are satisfied. From the convexity of the problem and the satisfaction of the KKT conditions, $(\widetilde{\bz},\widetilde{\blambda})=0$ is an optimal primal-dual solution of $P(\widetilde{\bd},\widetilde{\bG},\widetilde{\bp})$.
\end{proof}
{\color{blue} We highlight that the perfect prior guess is not a requirement (this is mentioned to highlight the consistency of the formulation). Specifically, Proposition \ref{prop:cons} is used to emphasize that the solution of coarse problem becomes close to the original solution as the prior guess of the solution becomes accurate.}

The following proposition establishes a condition that guarantees the existence of a solution for $P(\widetilde{\bd};\widetilde{\bG},\widetilde{\bp})$. 
\begin{proposition}\label{prop:feas}
  Consider $\bd\in\mathbb{D}$ and $\bz^o\in\mathbb{R}^{nN}$ that is feasible to $P(\bd;\bG,\bp)$, then there exists a solution of $P(\widetilde{\bd};\widetilde{\bG},\widetilde{\bp})$.
\end{proposition}
\begin{proof}
  Since $\bz^o$ is feasible, we have $\bG\bz^o \geq \bd$. By left multiplying $U^\top$, we have $U^\top \bG \bz^o \geq U^\top \bd$ (recall that $U\geq 0$). This implies $\widetilde{\bd} \leq 0$, so one can see that $\widetilde{\bz}=0$ satisfies $\widetilde{G} \widetilde{\bz} \geq  \widetilde{\bd}$. That is, $\widetilde{\bz}=0$ is feasible. The existence of a solution follows from feasibility and compactness of the feasible set.
\end{proof}

\subsection{Sensitivity to Coarsening}

We now analyze the sensitivity of the solution to coarsening. In other words, we estimate how the solution changes when the coarsening scheme is applied. Theorem \ref{thm:eds} will be used to establish the error bounds; to do so, we need to cast coarsening as a data perturbation. 

We begin by defining the notion of {\it free variables}. The index sets $\widetilde{\mathbb{S}}$ and $\mathbb{S}$ of free variables are defined as follows:
\begin{align*}
  \widetilde{\mathbb{S}}&:=\left\{k\in\mathbb{I}_{1:K}:|\mathbb{I}_{M_k:N_k}|=1 \text{ and }\lambda^o_i=0\right\}\\
  \mathbb{S}&:=\bigcup_{k\in\widetilde{\mathbb{S}}} \mathbb{I}_{M_k:N_k}.
\end{align*}
The set of free variables can be regarded as the set of variables that are not coarsened. The set $\widetilde{\mathbb{S}}$ is the associated index set on $\mathbb{I}_{1:K}$ and the set $\mathbb{S}$ is the associated index set on $\mathbb{I}_{1:N}$. We also define the following to quantify the coarsening error:
  \begin{align*}
    \Delta :=\max\left\{
    \max_{j\in\mathbb{I}_{1:N}}
    \Vert d_j-G_{j,j-1}z'_{j-1}-G_{j,j}z'_j\Vert
    : \bz' \in \mathbb{P} \right\},
  \end{align*}
  where $\mathbb{P}:=\{ \bz^o+ T\widetilde{\bz}: \widetilde{\bG}\widetilde{\bz} \geq \widetilde{\bd}\}$. The existence of the maximum directly comes from the compactness of $\mathbb{P}$. Note that $\Delta$ depends on the choice of $\{M_k\}_{k=1}^K$ and $(\bz^o,\blambda^o)$. 
  
We now state our main result.

\begin{theorem}[EDS to Coarsening]\label{thm:crs}
  Consider $\bd\in\mathbb{D}$ and $\bz^o\in\mathbb{R}^{nN}$ that is feasible to $P(\bd;\bG,\bp)$. There exists a solution $\bz^*=(z_1^*,\cdots,z_N^*)$ of $P(\bd;\bG,\bp)$ and a projected coarse solution $\bz'=(z'_1,\cdots,z'_N)$ of $P(\widetilde{\bd},\widetilde{\bG},\widetilde{\bp})$ such that the following holds for $i\in\mathbb{I}_{1:N}$.
  \begin{align}\label{eqn:crs}
    \Vert z_i^*-z'_i\Vert \leq \sum_{j\in \mathbb{I}_{1:N}\setminus\mathbb{S}}\Gamma \Delta\rho^{(|i-j|-1)_+}
  \end{align}
\end{theorem}

To prove Theorem \ref{thm:crs}, we first prove a technical lemma.
\begin{lemma}\label{lem:sens-to-cor}
  Let $\bz'$ be a projected coarse solution of $P(\widetilde{\bd},\widetilde{\bG},\widetilde{\bp})$. Then there exists a solution of $P(\bd'; \bG,\bp)$ and any of its solutions are projected coarse solutions of $P(\widetilde{\bd},\widetilde{\bG},\widetilde{\bp})$, where $\bd':=(d'_1,\cdots,d'_N)$ is defined by:
  \begin{align*}
    d'_i:=
    \begin{cases}
      d_i&\text{if } i\in\mathbb{S}\\
      G_{i,i-1}{z}'_{i-1}+G_{i,i}{z}'_i&\text{otherwise.}
    \end{cases}
  \end{align*}
\end{lemma}

\begin{proof}
  First we show that $\bz'$ is feasible to $P(\bd';\bG,\bp)$. Let $\widetilde{\bz}^*$ be the coarse solution associated with the projected coarse solution $\bz'$. The feasibility of $\widetilde{\bz}^*$ to $P(\widetilde{\bd},\widetilde{\bG},\widetilde{\bp})$ implies that $U^\top \bG T \widetilde{\bz}^* \geq U^\top \bd - U^\top \bG z^o$. Thus, $U^\top \bG\bz' \geq U^\top \bd$ holds. This implies:
  \begin{align}\label{eqn:eds-to-cor-0}
    \boldsymbol{G}_{k,k-1}\boldsymbol{z}_{k-1}'+\boldsymbol{G}_{k,k}\boldsymbol{z}_{k}' &\geq \boldsymbol{d}'_k
  \end{align}
  holds for $k\in\widetilde{\mathbb{S}}$. Furthermore, by the definition of $\bd'$, we have that \eqref{eqn:eds-to-cor-0} also holds for $k\in\mathbb{I}_{1:K}\setminus\widetilde{\mathbb{S}}$. This implies that $\bG\bz'\geq \bd'$. Thus, $\bz'$ is feasible to $P(\bd';\bG,\bp)$.

  Since $P(\bd';\bG,\bp)$ is feasible and the feasible set is compact, there exists a solution of $P(\bd';\bG,\bp)$. This proves the first part of the lemma. Let $\bz''$ be a solution of $P(\bd';\bG,\bp)$. Now we aim to show that $T^\top (\bz''-\bz^o)$ is a solution of $P(\widetilde{\bd},\widetilde{\bG},\widetilde{\bp})$.

  First we show the feasibility of $T^\top (\bz''-\bz^o)$ to $P(\widetilde{\bd},\widetilde{\bG},\widetilde{\bp})$. From the feasibility of $\bz'$ and $\bz''$ to $P(\bd';\bG,\bp)$, we have $\bG\bz' \geq \bd'$ and $\bG\bz''\geq \bd'$. By inspecting the definition of $\bG$, one can see every constraint \eqref{prob:lp-pre-coarse-con} for $k\in\mathbb{I}_{1:K}\setminus\widetilde{\mathbb{S}}$ effectively becomes equality constraint. Thus, the following holds:
  \begin{align}\label{eqn:eq-1}
    \begin{aligned}
      \boldsymbol{G}_{k,k-1}\boldsymbol{z}'_{k-1} + \boldsymbol{G}_{k,k}\boldsymbol{z}'_{k} &=  \boldsymbol{d}'_k\\
      \boldsymbol{G}_{k,k-1}\boldsymbol{z}''_{k-1} + \boldsymbol{G}_{k,k}\boldsymbol{z}''_{k} &=  \boldsymbol{d}'_k,  
    \end{aligned}
    \quad k\in\mathbb{I}_{1:K}\setminus\widetilde{\mathbb{S}}.
  \end{align}
  {\color{blue} From \eqref{prob:mpc-con-3}-\eqref{prob:mpc-con-4} and the fact that $U^\top_k$ effectively takes linear combinations over $i\in\mathbb{I}_{M_k:N_k}$, one can see that 
    \begin{align*}
      \bx''_k=\bx'_k,\quad \bu''_k=\bu'_k
    \end{align*}
    are included in $\bG_{k,k-1}\bz''_{k-1}+\bG_{k,k}\bz''_k=\bG_{k,k-1}\bz'_{k-1}+\bG_{k,k}\bz'_k$. This implies:}
  \begin{align}
    \boldsymbol{z}'_k &= \boldsymbol{z}''_k,\quad k\in\mathbb{I}_{1:K}\setminus\widetilde{\mathbb{S}}.\label{eqn:eq-2}
  \end{align}
  We have that for $k\in\widetilde{\mathbb{S}}$,
  \begin{align}\label{eqn:eq-3}
    T_kT_k^\top (\boldsymbol{z}_k''-\boldsymbol{z}^o_k) = \boldsymbol{z}''_k-\boldsymbol{z}^o_k,
  \end{align}
  {\color{blue} since $T_k$ is identity}. For $k\in\mathbb{I}_{1:K}\setminus\widetilde{\mathbb{S}}$,
  \begin{align*}
    T_kT_k^\top (\boldsymbol{z}_k''-\boldsymbol{z}^o_k) &= T_kT_k^\top (\boldsymbol{z}'_k-\boldsymbol{z}^o_k)\\
    &= T_kT_k^\top T_k \widetilde{z}^*_k\\
    &= T_k \widetilde{z}^*_k\\
    &= \boldsymbol{z}'_k-\boldsymbol{z}^o_k\\
    &= \boldsymbol{z}''_k-\boldsymbol{z}^o_k,
  \end{align*}
  where the first equality follows from \eqref{eqn:eq-2}, the second equality follows from the definition of projected coarse solution, the third equality follows from the orthogonality of $T_k$, the fourth equality follows from the definition of projected coarse solution, and the last equality follows from \eqref{eqn:eq-2}. This implies that \eqref{eqn:eq-3} also holds for $k\in\widetilde{\mathbb{S}}$. By the block diagonal structure of $T$, we can see that the following holds:
  \begin{align}\label{eqn:eds-to-cor-2}
    \bz''-\bz^o = TT^\top (\bz''-\bz^o).
  \end{align}
  By multiplying $U^\top$ to $\bG\bz'' \geq \bd'$ and using \eqref{eqn:eds-to-cor-2}:
  \begin{align}\label{eqn:ineq-1}
    U^\top \bG \left(TT^\top (\bz''-\bz^o) +  \bz^o\right)&\geq U^\top \bd'.
  \end{align}
  We have that from the definition of $\bd'$, the following holds:
  \begin{align}\label{eqn:ud-1}
    U_k^\top \boldsymbol{d}'_k=  U_k^\top \boldsymbol{d}_k,\; k\in\widetilde{\mathbb{S}}.
  \end{align}
  Moreover, for $k\in\mathbb{I}_{1:K}\setminus\widetilde{\mathbb{S}}$,
  \begin{subequations}\label{eqn:ud-2}
  \begin{align}
    U_k^\top \boldsymbol{d}'_k&= U_k^\top (\boldsymbol{G}_{k,k-1}\boldsymbol{z}'_{k-1} + \boldsymbol{G}_{k,k}\boldsymbol{z}'_{k}) \\
    &=\widetilde{G}_{k,k-1} \widetilde{z}^*_{k-1} +\widetilde{G}_{k,k} \widetilde{z}^*_k\\
    &\qquad+U_k^\top (\boldsymbol{G}_{k,k-1}\boldsymbol{z}^o_{k-1} + \boldsymbol{G}_{k,k}\boldsymbol{z}^o_{k})\nonumber\\
    &\geq U^\top_k \boldsymbol{d}_k ,
  \end{align}
  \end{subequations}
  where the second equality comes from the definition of projected coarse solution, and the inequality follows from the feasibility of $\widetilde{\bz}^*$ to $P(\widetilde{\bd},\widetilde{\bG},\widetilde{\bp})$. From \eqref{eqn:ud-1}-\eqref{eqn:ud-2}, we have that $U^\top \bd' \geq U^\top \bd$. This and \eqref{eqn:ineq-1} implies $\widetilde{\bG}T^\top (\bz''-\bz^o) \geq  \widetilde{\bd}$ and 
thus $T^\top (\bz''-\bz^o)$ is feasible to $P(\widetilde{\bd},\widetilde{\bG},\widetilde{\bp})$.

We now show that $T^\top (\bz''-\bz^o)$ is optimal to $P(\widetilde{\bd},\widetilde{\bG},\widetilde{\bp})$. From the feasibility of $\bz'$ to $P(\bd',\bG,\bp)$ we have that:
  \begin{align}\label{eqn:eds-to-cor-1}
    \bp^\top \bz''\leq \bp^\top \bz'.
  \end{align}
  From \eqref{eqn:eds-to-cor-2}, we have:
  \begin{subequations}\label{eqn:eds-to-cor-3.5}
    \begin{align}
      \bp^\top \bz''& =  \bp^\top (TT^\top (\bz''-\bz^o) + \bz^o ) \\
      &=(T^\top \bp)^\top T^\top (\bz''-\bz^o) + \bp^\top \bz^o \\
      \bp^\top \bz'& = (T^\top \bp)^\top \widetilde{\bz}^* + \bp^\top \bz^o.
    \end{align}
  \end{subequations}
  From \eqref{eqn:eds-to-cor-1}-\eqref{eqn:eds-to-cor-3.5}, we obtain:
  \begin{align}\label{eqn:eds-to-cor-4}
    (T^\top \bp)^\top T^\top (\bz''-\bz^o) \leq (T^\top \bp)^\top \widetilde{\bz}^*.
  \end{align}
  Next, we {\color{blue} compare} the objective values of $T^\top (\bz''-\bz^o)$ and that of $\widetilde{\bz}^*$ for $P(\widetilde{\bd},\widetilde{\bG},\widetilde{\bp})$.
  \begin{align}\label{eqn:eds-to-cor-5}
    &\widetilde{\bp}^\top T^\top (\bz''-\bz^o) - \widetilde{\bp}^\top\widetilde{\bz}^*\\\nonumber
    &= (T^\top \bp)^\top T^\top (\bz''-\bz^o) - (T^\top \bp)^\top \widetilde{\bz}^* \\\nonumber
    &\qquad - \left(T^\top \bG^\top \blambda^o \right)^\top \left(T^\top  (\bz''-\bz^o) - \widetilde{\bz}^* \right)\\\nonumber
    &\leq-\left(\bG^\top \blambda^o\right)^\top\left(\bz''-\bz' \right)\\\nonumber
    &\leq-(\blambda^o)^\top (\bG\bz''-\bG\bz') .
  \end{align}
  The first inequality follows from \eqref{eqn:eds-to-cor-4} and \eqref{eqn:eds-to-cor-2}. Recall that, by definition of $\widetilde{\mathbb{S}}$, $\boldsymbol{\lambda}^o_k = 0$ holds for $k\in\widetilde{\mathbb{S}}$.  From  \eqref{eqn:eq-1}, we have that
  \begin{align}\label{eqn:eq-4}
    \boldsymbol{G}_{k,k-1}\boldsymbol{z}'_{k-1} + \boldsymbol{G}_{k,k}\boldsymbol{z}'_{k} =
      \boldsymbol{G}_{k,k-1}\boldsymbol{z}''_{k-1} + \boldsymbol{G}_{k,k}\boldsymbol{z}''_{k}
  \end{align}
  holds for $k\in\mathbb{I}_{1:K}\setminus\widetilde{\mathbb{S}}$. This implies that $(\blambda^o)^\top (\bG\bz''-\bG\bz')=0$. Thus, \eqref{eqn:eds-to-cor-5} yields $\widetilde{\bp}^\top T^\top (\bz''-\bz^o)   - \widetilde{\bp}^\top\widetilde{\bz}^* \leq 0$.

  We observed that $T^\top (\bz''-\bz^o)$ is feasible to $P(\widetilde{\bd},\widetilde{\bG},\widetilde{\bp})$ and its objective is not greater than that of $\widetilde{\bz}^*$, which is a solution of $P(\widetilde{\bd},\widetilde{\bG},\widetilde{\bp})$. This implies that $T^\top (\bz''-\bz^o)$ is an optimal solution of $P(\widetilde{\bd},\widetilde{\bG},\widetilde{\bp})$. Therefore, $\bz^o+TT^\top (\bz''-\bz^o) = \bz''$ (here the equality follows from \eqref{eqn:eds-to-cor-2}) is a projected coarse solution of $P(\widetilde{\bd},\widetilde{\bG},\widetilde{\bp})$.
\end{proof}

Using Lemma \ref{lem:sens-to-cor}, one can compare $\bz'$ with $\bz^*$ by comparing $\bd'$ and $\bd$ using Theorem \ref{thm:eds}. In other words, coarsening can be regarded as a data perturbation. We now prove Theorem \ref{thm:crs}.

\begin{proof}[Proof of Theorem \ref{thm:crs}]
  By Proposition \ref{prop:feas} and the feasibility of $\bz^o$ to $P(\bd;\bG,\bp)$, there exists a solution of $P(\widetilde{\bd},\widetilde{\bG},\widetilde{\bp})$. We pick any of its projected coarse solution and construct $\bd'$ as given in Lemma \ref{lem:sens-to-cor}. From Lemma \ref{lem:sens-to-cor}, we have that $P(\bd';\bG,\bp)$ is feasible (thus $\bd'\in\mathbb{D}$). By Theorem \ref{thm:eds}, there exists a solution $\bz^*$ of $P(\bd;\bG,\bp)$ and a solution $\bz'$ of $P(\bd';\bG,\bp)$ such that:
  \begin{align*}
    &\Vert z^*_{i}(\bd)-z'_{i}( \bd')\Vert
    \leq \sum_{j=1}^N \Gamma\rho^{(|i-j|-1)_+}\Vert d_j- d'_j\Vert\\
    &\quad\leq \sum_{j\in\mathbb{I}_{1:N}\setminus\mathbb{S}}^N \Gamma\rho^{(|i-j|-1)_+}\left\Vert 
    d_j-G_{j,j-1}z'_{j-1} -G_{j,j}z'_j
    \right\Vert\\
    &\quad\leq \sum_{j\in\mathbb{I}_{1:N}\setminus\mathbb{S}}^N \Gamma \Delta\rho^{(|i-j|-1)_+},
  \end{align*}
  where the second inequality can be verified by inspecting the definition of $\bd'$ and the last inequality follows from the definition of $\Delta$. By Lemma \ref{lem:sens-to-cor}, $\bz'$ is a projected coarse solution of $P(\widetilde{\bd},\widetilde{\bG},\widetilde{\bp})$. 
\end{proof}

Recall that $i\in\mathbb{S}$ implies that the index $i$ is not coarsened. Accordingly, the index set $\mathbb{I}_{1:N}\setminus\mathbb{S}$ can be regarded as the set of time indexes that are coarsened. As such, each term in the right-hand side of \eqref{eqn:crs} represents the effect of coarsening. In receding horizon control, we are particularly interested in the error of $z_1'$, which is bounded as:
  \begin{align}\label{eqn:err1}
    \Vert z_1^*-z'_1\Vert\leq \sum_{j\in\mathbb{I}_{1:N}\setminus\mathbb{S}}^N \Gamma\Delta\rho^{(j-2)_+}.
  \end{align}

\subsection{Coarsening Strategies}\label{ssec:strategies}
We consider the following three coarsening strategies: equal-spacing scheme, full-then-sparse scheme, and diffusing-horizon scheme. Here we explain the procedure for creating the coarse grid for each strategy. Note that it suffices to explain the logic to choose $\{M_k\}_{k=1}^K\subseteq\mathbb{I}_{1:N}$. In the equal-spacing scheme, we place the points as:
  \begin{align*}
    M_k=\left\lfloor \dfrac{N(k-1)}{K}+1 \right\rfloor,\;k\in\mathbb{I}_{1:K}.
  \end{align*}
In the full-then-sparse scheme we choose the first $K$ points:
  \begin{align*}
    M_k&=k,\;k\in\mathbb{I}_{1:K}.
  \end{align*}
In the diffusing-horizon scheme we choose the points in a way that the spaces between the points increase exponentially:
  \begin{align*}
    M_k&=\max\left\{k,\left\lfloor (N+1)^{\frac{k-1}{K}}\right\rfloor\right\},\;k\in\mathbb{I}_{1:K}.
  \end{align*}
  The different coarsening schemes are illustrated in Fig. \ref{fig:strategies}. Theorem \ref{thm:crs} serves as a guiding principle for designing the coarse grid. Theorem \ref{thm:crs} implies that the farther away the coarsened indexes are placed, the more accurate the projected coarse solution $z'_1$ is. This provides a justification that inducing a monotonically increasing sparse grid is a good coarsening strategy. Thus, full-then-sparse and diffusing-horizon schemes are good strategies. However, the quantity $\Delta$ depends on the choice of coarsening scheme, but such a dependency is difficult to characterize. As such, Theorem \ref{thm:crs} does not allow quantitative comparison between coarsening strategies. In Section \ref{sec:cstudy} we perform numerical experiments to compare the practical effectiveness of the different coarsening strategies.

The feasibility of $z'_1$ (in the sense that $G_{1,1}z'_1=v_1$ holds so that it can be implemented) may not be guaranteed if $1\notin\mathbb{S}$. In such a case, one can modify the coarsening scheme by always selecting $M_2=2$, and then applying the coarsening scheme to $\mathbb{I}_{2:N}$ to obtain $M_3,\cdots,M_K$. In this way, the feasibility of $z'_1$ can always be guaranteed. 

{\color{blue}  EDS has been established for general nonlinear MPC in \cite{na2020exponential} under strong second order conditions. This EDS result holds {\it locally} (it is only guaranteed in a neighborhood of the base solution); accordingly, EDS will only hold if coarsening induces a parametric perturbation that is in a neighborhood of the solution. Specifically, for nonlinear problems, it does not seem possible to establish the error bound of Theorem \ref{thm:crs}. In practice, however, we expect that exponential coarsening will be effective. We also emphasize that establishing EDS in a linear programming setting is difficult because there is no notion of curvature (second order conditions) and thus we required a different analysis procedure than that used in \cite{na2020exponential}. The proposed diffusing-horizon formulation can be seen as a suboptimal MPC scheme; as such, we expect that existing results in suboptimal MPC \cite{pannocchia2011conditions,allan2017inherent} can be used to analyze stability and recursive feasibility, but a rigorous analysis is left as a topic of future work.}

\section{Numerical Case Study}\label{sec:cstudy}
We have applied the coarsening strategies developed in Section \ref{ssec:strategies} to design MPC controllers for a central HVAC plant. The central plant seeks to satisfy time-varying energy demands from a university campus by manipulating chillers, storage tanks, and transactions with utility companies. The problem details are provided in \cite{kumar2019stochastic}. The system under study is illustrated in Fig. \ref{fig:campus}. The problem used here is a simplified version where the peak demand cost is eliminated. 

\begin{figure*}[t]
  \centering
  \includegraphics[width=0.6\textwidth]{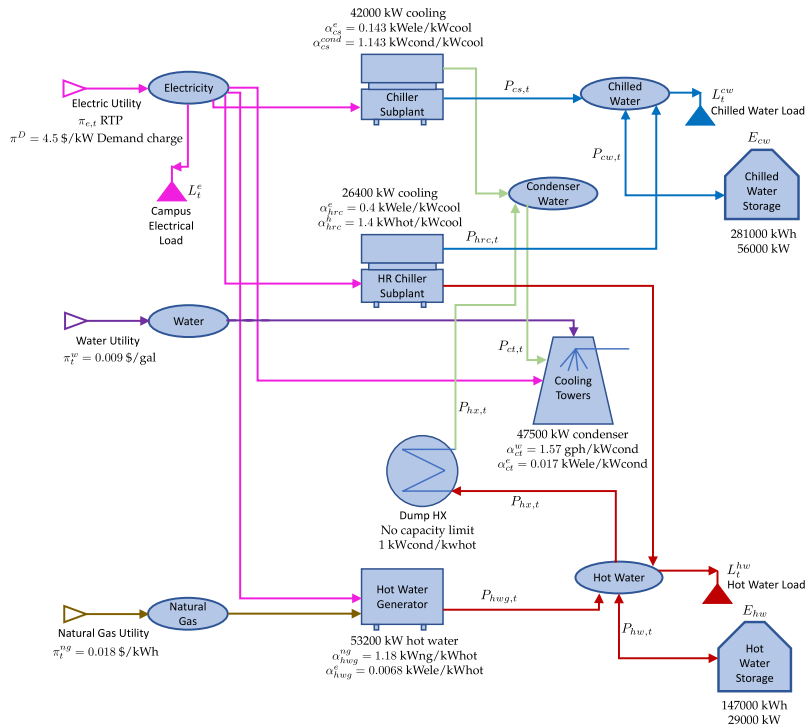}
  \caption{Schematic representation of central HVAC plant under study.}\label{fig:campus}
\end{figure*}

\subsection{Problem and Data Description}
The problem under study has the form \eqref{prob:mpc}. A description of the state $x_i$ and control $u_i$ variables and their upper and lower bounds are provided in Table \ref{tbl:variable}. Here, CW denotes chilled water and HW denotes hot water. {\color{blue} The description of the stage data $q_i$, $r_i$, $v_i$, $w_i$ is provided in Table \ref{tbl:data} (the data entries that are not mentioned are constantly zero)}. Time profiles for the problem data are shown in Fig. \ref{fig:data}. The description of the constraints (each row of \eqref{prob:mpc-con-1}-\eqref{prob:mpc-con-2}) is shown in Table \ref{tbl:constraint}. Here $A1-A2$ refer to the constraints associated with \eqref{prob:mpc-con-1} and $E1-E6$ refer to the constraints associated with \eqref{prob:mpc-con-1}. The problem statistics (including $\Delta t$, $N$, and $K$) are shown in Table \ref{tbl:stat}. Since feasibility cannot be guaranteed if $L_1\neq 1$, we use a modified coarsening scheme (as discussed in Section \ref{ssec:strategies}).  We also use zero prior solution guess $(\bz^o,\blambda^o)=0$. The problems are solved with the commercial optimization solver {\tt Gurobi} \cite{gurobi} and modeled with algebraic modeling language {\tt JuMP} \cite{dunning2017jump}. The case study is run on Intel Xeon CPU E5-2698 v3 running at 2.30GHz.



\begin{table}
  \centering
  \caption{Variable specifications}\label{tbl:variable}
  \begin{tabular}{|c|l|c|c|}
    \hline 
    Variables & Specifications \\
    \hline
    $x_i[1]$& Energy level of the CW storage [kWh] \\
    \hline
    $x_i[2]$& Energy level of the HW storage [kWh]\\
    \hline
    $u_i[1]$& Load of chiller subplant [kW] \\
    \hline
    $u_i[2]$& Load of heat recovery chiller subplant [kW] \\
    \hline
    $u_i[3]$& Load of HW generator [kW] \\
    \hline
    $u_i[4]$& Load of the cooling towers [kW] \\
    \hline
    $u_i[5]$& Charge/discharge of the CW storage [kW]\\
    \hline
    $u_i[6]$& Charge/discharge of the HW storage [kW] \\
    \hline
    $u_i[7]$& Load of the dump heat exchanger [kW] \\
    \hline
    $u_i[8]$& Electricity demand [kW] \\
    \hline
    $u_i[9]$& Water demand [gal] \\
    \hline
    $u_i[10]$& Natural gas demand [kW] \\
    \hline
    $u_i[11]$& Slack (pos.) for unmet CW load [kW] \\
    \hline
    $u_i[12]$& Slack (neg.) for unmet CW load [kW] \\
    \hline
    $u_i[13]$& Slack (pos.) for unmet HW load [kW] \\
    \hline
    $u_i[14]$& Slack (neg.) for unmet HW load [kW] \\
    \hline
  \end{tabular}
\end{table}

\begin{table}
  \centering
  \caption{Data specifications}\label{tbl:data}
  \begin{tabular}{|c|l|c|}
    \hline 
    Data & Specifications & Values\\
    \hline
    $r_i[8]$& Electricity price [\$/kWh] & Fig. \ref{fig:data}\\
    \hline
    $r_i[9]$& Price of water [\$/gal] &0.009\\
    \hline
    $r_i[10]$& Price of natural gas [\$/kWh] &0.018\\
    \hline
    $r_i[11]$& Penalty for unmet (pos.) CW load [\$/kWh] &  45\\
    \hline
    $r_i[12]$& Penalty for unmet (neg.) CW load [\$/kWh] &45\\
    \hline
    $r_i[13]$& Penalty for unmet (pos.) HW load [\$/kWh] &45\\
    \hline
    $r_i[14]$& Penalty for unmet (neg.) HW load [\$/kWh] &45\\
    \hline
    $w_i[1]$& Electrical load of campus [kWh] & Fig. \ref{fig:data} \\
    \hline
    $w_i[5]$& CW load [kWh] & Fig. \ref{fig:data}\\
    \hline
    $w_i[6]$& HW load [kWh] & Fig. \ref{fig:data}\\
    \hline
  \end{tabular}
\end{table}

\begin{figure}
  \centering
  \includegraphics[width=0.45\textwidth]{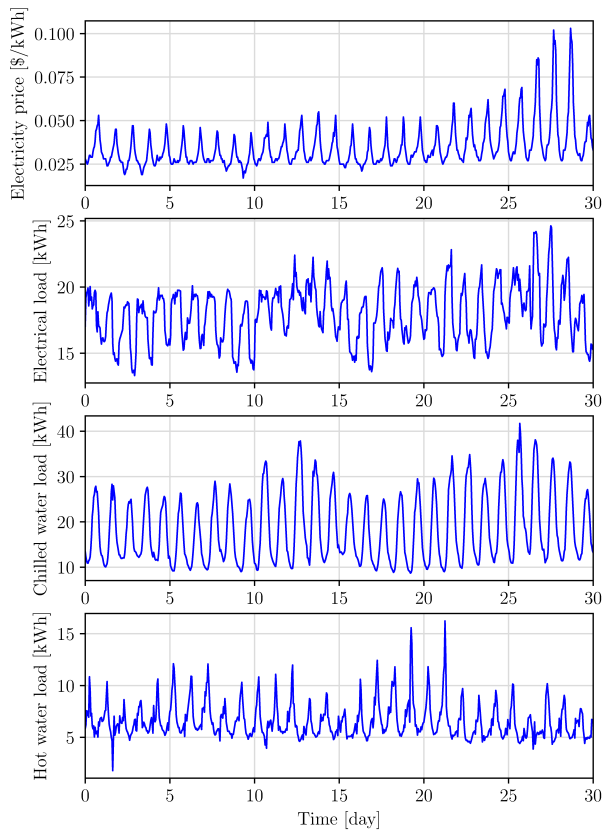}
  \caption{Time-varying data for HVAC plant problem. Electricity price $r_i[8]$, electricity load $w_i[1]$, CW load $w_i[5]$, and HW load $w_i[6]$.}  \label{fig:data}
\end{figure}

\begin{table}
  \centering
  \caption{Constraint specifications}\label{tbl:constraint}
  \begin{tabular}{|c|l|l|l|l|}
    \hline 
    Constraints & Specifications\\
    \hline
    A1& CW storage balance  \\
    \hline
    A2& HW storage balance  \\
    \hline
    E1& Electricity balance \\
    \hline
    E2& Water balance \\
    \hline
    E3& Natural gas balance \\
    \hline
    E4& Energy balance of condenser water \\
    \hline
    E5& CW load balance \\
    \hline
    E6& HW load balance \\
    \hline
  \end{tabular}
\end{table}

\begin{table*}
  \centering
    \caption{Summary of problem statistics}\label{tbl:stat}
  \begin{tabular}{|l|c|c|c|}
    \hline
    & Sensitivity Analysis
    & Solution Trajectory Analysis
    & Closed-Loop Analysis\\
    \hline 
    Time step length (full) $\Delta t$ &1 min&1 min& 1 min\\
    \hline
    Prediction horizon &1 day&  1 week& 1 week\\
    \hline
    Simulation horizon &-& -& 1 month\\
    \hline
    Number of time steps (full) $N$ &1,440& 10,080 &10,080\\
    \hline
    Number of variables in LP (full)&27,346 &  191,506&191,506\\
    \hline
    Number of equality constraints in LP (full) & 11,515& 80,635 &80,635 \\
    \hline
    Number of time steps (coarse) $K$ &-& 101 &101\\
    \hline
    Number of variables in LP (coarse)&-& 1,905&1,905\\
    \hline
    Number of equality constraints in LP (coarse) &-& 803& 803 \\
    \hline
  \end{tabular}
\end{table*}

\subsection{Sensitivity Analysis}
We first numerically verify that EDS (Theorem \ref{thm:eds}) holds. Here, we verify the decay of sensitivity by obtaining samples of perturbed solutions. First, we formulate problem $P(\bd)$ with the original data $\bd$. We refer to the solution of $P(\bd)$ as the reference solution. We then obtain samples of the {\it perturbed} problem $P(\bd+\bdelta)$ where $\bdelta$ is a random perturbation drawn from the  multivariate normal distribution $\bdelta_i[k]\sim N(0,\sigma_k^2)$. In particular, we use $\sigma_k=10$ for $k\in\{3,7,8\}$ and $\sigma_k=0$ otherwise (i.e., we only perturb electrical loads, CW loads, and HW loads). The solutions of the perturbed problems are called perturbed solutions ({\color{blue} the random perturbations are introduced to the forecast data $\{d_i\}_{i=1}^N$ within Problem \eqref{prob:mpc}}). To assess the decay of sensitivity, we apply the perturbation only in specific locations in the time horizon. In particular, for each case study (i)-(iv), we applied perturbations $d_j\leftarrow d_j+\delta_j$ only at time locations (i) $0\leq j\leq 99$, (ii) $100\leq j\leq 199$, (iii) $200\leq j\leq 299$, and (iv) $300\leq j\leq 399$ to obtain the samples of perturbed solutions $\bz^*(\bd+\bdelta)$. For each case, $1000$ perturbation samples are obtained. Parts of the samples $(u_1[5],u_1[6])$ are shown in Fig. \ref{fig:sens}. Random noise from $N(0,\diag(500,500))$ is added to the samples of $(u_1[5],u_1[6])$ to enhance the visibility of overlapped points (the variances are appropriately set so that the random noise does not distort the overall shape of the distribution).

\begin{figure}[t]
  \centering
  \includegraphics[width=.45\textwidth]{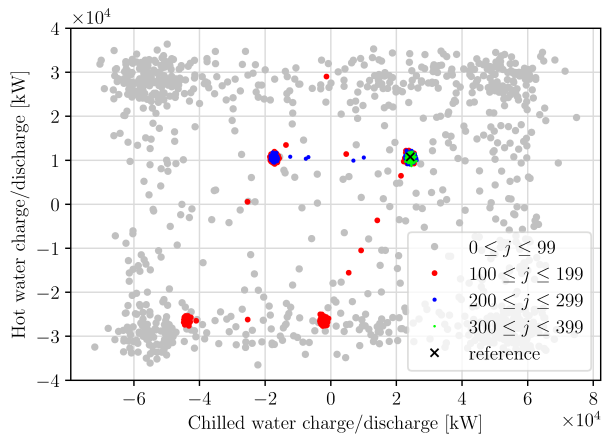}
  \caption{Effect of perturbations at different time locations on controls $(u_1[5],u_1[6])$. When the perturbation location $j$ is small the effect of the perturbation is strong and the effect decays as $j$ becomes larger.}\label{fig:sens}
\end{figure}

From the results, we can see that the sensitivity of the effect of perturbation indeed decays with respect to the distance from the perturbation. Specifically, one can see that the samples obtained from perturbations on time indexes $0\leq j\leq 99$ (case (i)) cover a large region around the reference solution (the solution is highly sensitive). On the other hand, the region of the samples obtained from perturbations on $100\leq j\leq 199$ is much smaller (case (ii)) and this effect can also be observed in cases (iii)-(iv) (solutions are less sensitive).

\subsection{Solution Trajectory Analysis}
In Fig. \ref{fig:sol}, we show the solution trajectories obtained with the coarse problems for the three different coarsening schemes discussed in Section \ref{ssec:strategies} along with the solution trajectory obtained with the full-resolution scheme. Here, we only show the trajectories of $z'_i[1]$ (the CW storage). One can see that the coarse trajectories capture the overall behavior of the full-resolution solution. In particular, the equal-spacing scheme most closely resembles the full-resolution trajectory. However, as we discussed before, the solution quality of the far future locations does not have a big impact on the implemented control action. Thus, this does not necessarily mean that the equal-spacing scheme has the best closed-loop performance. Similarly, even if the full-then-sparse scheme seemingly has bad performance, the closed-loop performance may actually be better. We can also see how the solution of the diffusing horizon scheme becomes increasingly coarse as time proceeds.

\subsection{Closed-Loop Analysis}\label{ssec:perf}
We evaluated the closed-loop economic performance of the MPC controller for three different coarsening strategies. The performance of the full-resolution problem (without coarsening) was also assessed and used as a reference. The closed-loop economic performance is calculated as $\sum_{i=1}^{N_{\text{sim}}} p_i^\top z'_i$.
It is important to note that the actions $z'_i$ are obtained from the solution of MPC problems solved at different sampling times. A comparison of the closed-loop costs is shown in Fig. \ref{fig:costs} as cumulative frequency plots with $10$ different scenarios. Each scenario uses a different data sequence (drawn from historical data).  We observe that, among the three coarsening schemes, the diffusing-horizon scheme always performed best. In particular, the performance of the diffusing-horizon scheme was close to the performance of the full-resolution controller. On average, we observed a $3\%$ increase in the closed-loop cost while the equal-spacing and full-then-coarse schemes experienced cost increases of $350\%$ and $130\%$, respectively.  The computational times for the MPC problems (for the entire closed-loop simulation) are compared in Fig. \ref{fig:times}. We observe that the three coarsening strategies significantly reduce times (by two orders of magnitude). Specifically, total solution times were reduced from hours to two minutes. We can thus conclude that the diffusing-horizon coarsening scheme can effectively reduce computational complexity while maintaining solution quality. 

\begin{figure}
  \centering
  \includegraphics[width=.45\textwidth]{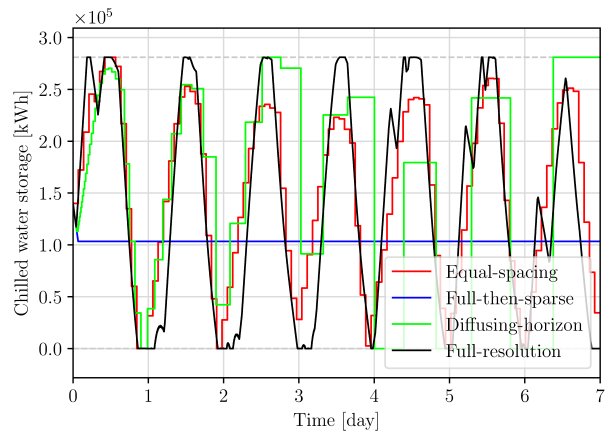}
  \caption{Solution trajectories for different coarsening schemes.}\label{fig:sol}
\end{figure}

\begin{figure}
  \centering
  \includegraphics[width=.45\textwidth]{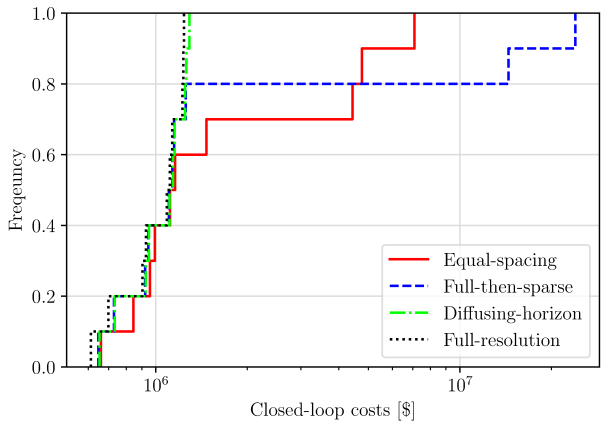}
  \caption{Closed-loop economic performance for different coarsening schemes.}\label{fig:costs}
\end{figure}

\begin{figure}
  \centering
  \includegraphics[width=.45\textwidth]{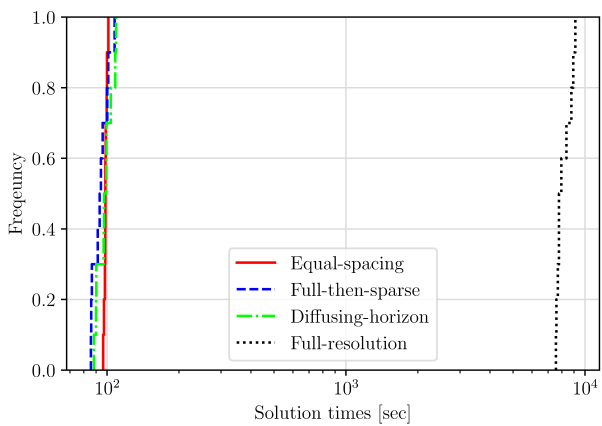}
  \caption{Total solution times for different coarsening schemes.}\label{fig:times}
\end{figure}

\section{Conclusions}
We have presented an algebraic coarsening strategy that we call diffusing-horizon MPC. The strategy is motivated by a fundamental solution property of MPC problems that indicates that the effect of data perturbations at a future time decays exponentially as one moves backward in time. Using the MPC case study for the central HVAC plant, we have demonstrated that the suggested strategy can effectively reduce the computational complexity of the problem while maintaining high economic performance. We highlight that the application of the diffusing-horizon strategy is not confined to time-domain coarsening but it can also be applied to general graph structures. A particularly interesting application of diffusing horizons would be to conduct scenario tree aggregation for multi-stage stochastic MPC problems, whose complexity grows much faster than deterministic MPC problems. These are interesting directions of future work.

\section{Acknowledgements}

We acknowledge partial funding from the National Science Foundation under award NSF-EECS-1609183.

\bibliographystyle{IEEEtran}
\bibliography{diffusing_horizon_mpc}

\vspace{-0.2in}
\begin{IEEEbiographynophoto}{Sungho Shin}  received the B.S. degree in chemical engineering and mathematics from Seoul National University, Seoul, South Korea, in 2016. He is currently working toward the Ph.D. degree with the Department of Chemical and Biological Engineering, University of Wisconsin-Madison, Madison, WI, USA. His research interests include control theory and optimization algorithms for complex networks.
\end{IEEEbiographynophoto}
\vspace{-0.2in}
\begin{IEEEbiographynophoto}{Victor M. Zavala}
  is the Baldovin-DaPra Associate Professor in the Department of Chemical and Biological Engineering at the University of Wisconsin-Madison. He holds a B.Sc. degree from Universidad Iberoamericana and a Ph.D. degree from Carnegie Mellon University, both in chemical engineering. He is an associate editor for the Journal of Process Control and a technical editor of Mathematical Programming Computation. His research interests are in the areas of energy systems, high-performance computing, stochastic programming, and predictive control.
\end{IEEEbiographynophoto}
\end{document}